\documentclass{article}
\usepackage[english]{babel}
\usepackage{amsmath,amsfonts,amsthm,amssymb,amscd}
\usepackage{graphicx}

 \usepackage{verbatim}

\newcommand{\p}{\partial}
\newcommand{\e}{\varepsilon}
\newcommand{\al}{\alpha}

\newcommand{\R}{{\mathbb R}}

\newcommand{\N}{{\mathbb N}}

\newcommand{\aA}{{\cal A}}

\newcommand{\FF}{{\cal F}}

\newcommand{\HH}{{\cal H}}

\newcommand{\sS}{{\cal S}}

\newcommand{\XX}{{\cal X}}
\newcommand{\YY}{{\cal Y}}

\newcommand{\ri}{{ \to }}

\newcommand{\ph}{\varphi}
\newcommand{\Z}{ \mathbb{Z}}
\newcommand{\dd}{{\textup d}}

\newcommand{\supp}{\mathop{\rm supp}\nolimits}
\newcommand{\diver}{\mathop{\rm div}\nolimits}
\newcommand{\rot}{\mathop{\rm curl}\nolimits}
\newcommand{\curl}{\mathop{\rm curl}\nolimits}

\theoremstyle{plain}
\newtheorem{theorem}{Theorem}[section]
\newtheorem{lemma}[theorem]{Lemma}
\newtheorem{proposition}[theorem]{Proposition}

\newtheorem*{main result}{Main result}
\newtheorem{definition}[theorem]{Definition}

\theoremstyle{remark}

\newtheorem{remark}[theorem]{Remark}

\numberwithin{equation}{section}
\begin{document}

\author{Hayk Nersisyan}
\date{}
\title{ Stabilization of the 2D incompressible  Euler system in an  infinite strip }

 \maketitle\begin{center}

CNRS UMR 8088, D\'epartement de Math\'ematiques\\
Universit\'e de Cergy--Pontoise, Site de Saint-Martin\\
            2 avenue Adolphe Chauvin\\
       F95302 Cergy--Pontoise Cedex, France\\
       E-mail: Hayk.Nersisyan@u-cergy.fr
 \end{center}

\vspace{15 pt}

{\small\textbf{Abstract.} The paper is devoted to the study of a
 stabilization problem for the  2D incompressible Euler system  in an  infinite strip
 with boundary controls. We show that for any stationary solution $(c,0)$ of the Euler system
 there is a control which is supported in a given bounded part of the boundary of the strip and stabilizes the system to   $(c,0)$.
   }\\\\
 \tableofcontents
 \newpage
 \section{Introduction}
We consider the incompressible two-dimensional Euler system
\begin{align}\label{I.1}
\dot{u}+\langle u,\nabla\rangle u+\nabla p=0,\quad  \diver u=0 ,
\end{align}
where $u=(u_1,u_2)$ and $p $  are unknown velocity field and
pressure of the fluid, and
\begin{equation}
\langle u,\nabla\rangle
v=\sum_{i=1}^3u_i(t,x)\frac{\p}{\p{x_i}}v.\nonumber
\end{equation}
The space variable $x=(x_1,x_2)$ belongs to the strip $D$ defined by
 \begin{align} \label{0:cyl}
D:=\{(x_1,x_2): \, x_1\in \R, x_2\in(-1,1)\}.
\end{align}
 Let us take two open intervals $ (a,b), (a+d,b+d)\subset \R$    and denote
\begin{align}\label{0:gammaisah1}
\Gamma_0=  (a,b) \times \{1\} \cup(a+d,b+d)  \times \{-1\}.
\end{align}
The aim of this paper is the study of stabilization of (\ref{I.1})
with boundary controls supported by $\Gamma_0$. System (\ref{I.1})
is completed with the boundary and initial conditions
\begin{align}
u\cdot n&=0  \text{ on }  \Gamma\setminus \Gamma_0  ,\label{I:3}\\
u(x,0)&=u_0(x),\label{I:init}
\end{align}
where $\Gamma:=\p D$ and $n$ is the outward unit normal vector
on~$\Gamma$. In particular, (\ref{I:3}) is equivalent to $ u_2=0$
on $\Gamma\setminus \Gamma_0$.

  For any  integer $s\ge0$  we denote by  $H^s  (D  )$ the
space of vector functions $u = (u_1, u_2)$ whose components belong
to the Sobolev space of order~$s$ and by $\|\cdot \|_{s,D}$  the
corresponding norm. If there is no confusion, we drop the index
$D$. In the case $s=0$, we write $\|\cdot\|:=\|\cdot\|_0.$ For any
integer  $s>0$ we define $\HH^s(D)$ as the space of distributions
$u$ in $D$ with $\nabla u\!\in~ \!H^{s-1}(D)$. We equip $\HH^s(D)$
with the semi-norm
$$ \|u \|_{\HH^s(D)}:= \|\nabla u \|_{ s-1}.$$
 We denote  by
$\dot H ^s (D)$  the   quotient space $\HH^s(D)/ \R$. The
following theorem is our main result.
\begin{main result}\label{1:T.Int} Let $ s\ge4 $ be an integer.
Then for any   constant $c\in \R$ and initial function $u_0\in ~\!
H^s(D)$ that  decays fast at infinity and satisfies the relations
\begin{align*}
 \diver u_0 =0, \quad u_0\cdot n =0   \text{ on }  \Gamma\setminus\Gamma_0
\end{align*}
  there exists a solution
$(u,p) \in C(\R_+, C(\overline D)\cap \dot H^s(D))\times C(\R_+,
\dot H^{s}(D)) $  of (\ref{I.1}), (\ref{I:3}) and (\ref{I:init})
such that
\begin{align*}
 \lim_{t \rightarrow \infty }(\|u(\cdot,t)-(c,0)\|_{L^{\infty}}+ \|\nabla u\|_{ s-1  }+ \|\nabla p\|_{ s-1  }) =0.
 \end{align*}
\end{main result}
For the exact statement see Theorem \ref{T:himnth}. In this
formulation the control is not given explicitly, but we can assume
that control acts on the system as a boundary condition on
$\Gamma_0$.

Before turning to the ideas of the proof, let us describe in a few
words some previous results on the controllability of Euler and
Navier--Stokes systems. Coron \cite {cor} introduced the return
method to show exact boundary controllability of 2D incompressible
Euler system in a bounded domain. Glass \cite {gla} generalized
this result for 3D Euler system. Chapouly \cite{chap}  using
return method proved the global null controllability of the
Navier--Stokes system in rectangle. Recently, Glass and  Rosier
\cite{glro} proved the controllability of the motion of a rigid
body, which is surrounded by an incompressible fluid.
Controllability of Euler and Navier--Stokes systems with
distributed controls is studied in \cite{agr1,  fuim,  hn, shi1};
see also the book \cite{Corbook} for further references.

Notice that the above papers concern the problem of
controllability of the fluid in a bounded domain. In this paper,
we develop   Coron's return method to get the controllability of
velocity of 2D Euler system in an unbounded strip. This method
consists in reducing the controllability of nonlinear system to
the linear one. To this end, one  constructs a particular solution
$ (\overline u, \overline p)$ of Euler system and a sequence of
balls $\{B_i\}$ covering  $\overline D$, such that
 \begin{itemize}
   \item   [$(P)$] Any ball $B_i$ driven by the flow of $\overline
u $ leaves $\overline D$ through $ \Gamma _0$ at some time.
     \end{itemize}
Then the linearized
  system around  $\overline
u $ is controllable.  In our case, since the domain $D$ is
unbounded, the number of balls $B_i$ is infinite, thus we cannot
construct a bounded function $\overline u$, whose flow moves all
balls outside $D$ in a finite time. However, we can find a
particular solution $\overline u$ such that property $(P)$ holds
in infinite time. This proves the  stabilization of linearized
system in infinite time.

To show that controllability of linearized system implies  that of
the nonlinear system, we need to prove that $(P)$ also holds for
any $\tilde u$    sufficiently close to  $\overline u$. This is
obvious in the case of bounded domain. In our case, to prove this,
we need some additional properties for $\overline u$. In
particular, we need to construct a solution $\overline u$,
 which decays at infinity
faster than  $ 1/ x_1^2$. As our particular solution $\overline u$
is a combination of the Green functions of the Laplacian with
Neumann boundary condition,  we need to  prove that  Green
functions  decay at infinity. This property is a consequence of
elliptic regularity and some explicit formulas for solutions of
the Laplace equation in a strip.

The paper is organized as follows. In Section \ref{S:1}, we give preliminaries on Poisson and Euler equations in an unbounded
 strip.  The main results of the paper are
presented in  Section  \ref{S:2}. In  Section \ref{S:3},  we
construct the particular solution $\overline u$. In the Appendix,
we prove an auxiliary result used in Section \ref{S:1}.

\textbf{Acknowledgments.}   The author would like to express deep
gratitude to Armen Shirikyan for drawing his attention to this
problem and  for many fruitful suggestions and also to Nikolay Tzvetkov
for useful remarks on the Euler system.
\\
\newline

 \textbf{Notation.}

Let $J_T:=[0, T )$.   The space of continuous functions $u: J_T
\rightarrow X$ is denoted by $C(J_T,X )$. For any integer $s\ge0$
or $s=\infty$, we denote  $$C^s_b(D) = \{u \in C^s(D) :\,\,
\|u\|_{L^\infty(D)} < \infty\}.$$ We set  $\dot H^\infty (D
):=\cap_{s=0}^\infty \dot H^s (D )$. Define $$ \sS(D):=\{ u\in
L^2(D):\ x_1^\al \p ^\beta u(x_1,x_2)\in L^2(D)\text{ for any }
\al\in \R_+, \beta\in \Z_+^2 \}.$$ For a vector field $ u = (u_1,
u_2)$ we set
$$\curl u= \p_1u_2-\p_2u_1. $$
  The interior of a set $K$ is denoted by $int(K)$. Let
$B(x_0,r)$ be the closed ball in $\R^2$ of radius $r$ centred at
$x_0$. We denote by $C$ a universal constant whose value may
change from line to line.

 \section{Preliminaries}\label{S:1}
In this section, we present some auxiliary results on Poisson and Euler equations in an unbounded
 strip. The methods used in
their proofs are well known and in many cases we confine ourselves
to a brief description of the main ideas.

 \subsection{Poisson equations in an unbounded strip}\label{S:2.1}

First, let us describe the spaces $\dot H^s(D)$.
\begin{proposition}\label{P:Hikaruc} For any integer $s \ge 1$ we
have
\begin{itemize}
  \item  [(i)]  The space  $\dot H^s(D)$ is complete.
   \item   [(ii)]
$
  \HH^s(D) \, =\{u\in H^s_{loc}(D):\, \nabla u\in H^{s-1} \}.
$
   \item   [(iii)] If $s\ge 3$, then for any $u\in \HH^s(D)$ there is a constant $C$ depending on $u$ such that
   $$ |u(x_1,x_2)|\le C|x_1|+C$$
 holds for  all $x\in D$.
     \end{itemize}
\begin{proof} Let $ \{u_n\} \subset \dot H^s(D)$ be a Cauchy sequence. Then there is $v\in H^{s-1}(D)$ such that $ \nabla
u_n\ri v$ in $H^{s-1}(D)$ as $n\ri \infty$, and for any $\ph\in
C_0^\infty(D)  $ such that $ \diver \ph=0$, we  have
\begin{align*}
0= \lim _{n\ri \infty}(\nabla u_n,\ph)_{L^2}=(v,\ph)_{L^2}.
\end{align*}
 Hence,
$v=\nabla z$, where $z\in \dot H^s(D)$. This proves that $\dot
H^s(D)$ is complete. Now let us prove assertion $(ii)$.  Clearly
the space in the right-hand side  is contained in $\dot H^s(D)$.
Let us take a function $u\in \dot H^s(D)$, a compact set $K\subset
D$ and let us show that $ u\in H^s(K)$. Take two functions
$\chi,\chi_1\in C_0^\infty(D)$ and a compact  set $K_1\subset D$
with $ int( K_1 )\supset {K} $ such that $\chi =1$ in $K_1$ and $
\chi_1=1$ in $ \tilde K_1:=\supp\chi$. Then there exists $r\in \N$
such that $\chi_1 u\in H^{-r}(D)$. This implies that $ u\in
H^{-r}(\tilde K_1)$, hence
$$\Delta(\chi u )=2\nabla\chi\nabla u +\chi\Delta u+ u\Delta\chi \in H^{\min(-r;s-2)}( \tilde
K_1).$$ The elliptic regularity implies   $\chi u \in
H^{\min(-r+2;s)}(D)$, thus   $ u~\!\!\in~\!\!
H^{\min(-r+2;s)}(K_1)$. Repeating this argument for a compact set
$ K_2\subset  K_1$ with $ int (K_2) \supset    {K} $  we can show
that $u \in H^{\min(-r+4;s)}(K_2)$. Iterating this, we get $u \in
H^{s}(K)$. This completes the proof of assertion $(ii)$.

It is easy to see      that $(ii)$ implies $(iii)$. Indeed, from
$(ii)$ we get $$ u(x_1,x_2)=\int_0^{x_1}\p_1 u(y,x_2)\dd
y+u(0,x_2).$$ The Sobolev inequality yields $(iii)$.
\end{proof}

\end{proposition}

Now we summarize some facts about  Poisson equation. Let us take a
non-negative function $\gamma\in C_0^\infty(\R)$ such that $\supp
\gamma = [a,b]$ and $ \gamma\neq0 $ in $(a,b)$ and define
 \begin{align}
\tilde D:=\{ (x_1,x_2):\quad x_1\in \R, \,\,x_2\in (
-1-\gamma(x_1-d),1+\gamma(x_1) )\} \label{0:0:dtildisahm}
\end{align}
  (see figure 1).

\begin{figure}[htbp]
\begin{center}
\includegraphics{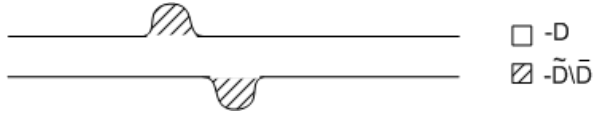}
\end{center}
\caption{\footnotesize Domain $\tilde D$}
\end{figure}

 Let us take $D'=D$ or $D'=\tilde D$ and consider the Dirichlet problem for the
Poisson equation:
\begin{align}
\Delta u=  f\,\,\, &\text{ in}\quad D',\label{0:11121}\\
 u =0 \,\,\,&\text{ on} \quad
\Gamma',\label{0:21121}
\end{align}
where  $\Gamma'=\p D'$ and $f\in  L^{2}({D'})$. We say that $u \in
H^1_0({D'}) $ is a solution of  (\ref {0:11121}), (\ref {0:21121})
if
\begin{align*}
\int_{D'} \nabla u \nabla \theta \dd x =-\int_{D'} f   \theta \dd x
\end{align*}
for any $\theta \in  H^1_0({D'}) $. We have  the following result for the  well-posedness of this problem.
\begin{proposition}\label{T:poisMazDir} For any integer $s \ge 0$ and for any  $ f\in  H^{s}({D'})$
problem (\ref {0:11121}), (\ref {0:21121}) has a unique solution
$u\in H^{s+2}({D'})  $. Moreover,
\begin{align}\label{E:poisdirichgn}
\| u\|_{ {s+2}  } \le C \|f\|_s,
\end{align}
where $C$ depends only on $s$.
\end{proposition}
\begin{proof}The existence of the solution $u \in   H^1_0({D'}) $ is a consequence of the Riesz representation theorem.
Clearly, we have
\begin{align}\label{E:poisdirichgn11}
\|\nabla u\| ^2 \le C \|f\|\| u\| .
\end{align}
The Poincar\'{e}  inequality applied to $u(x_1,\cdot)$ gives
$$\|u\|\le C \|\p_2 u\|.$$
Combining this with (\ref{E:poisdirichgn11}), we obtain
\begin{align}\label{E:poisdirichgn1121}
\| u\|_{ 1  } \le C \|f\|.
\end{align}
To show the regularity of the  solution and estimate (\ref{E:poisdirichgn}), we need the following lemma.
\begin{lemma}\label{L:rotdivihamar} For any integer $s\ge 1$ we have
\begin{align*}
 H^s\!(D') \! =\!\{z\!\in\! L^2(D'):\, \!\rot z \!\in \!H^{s-1}(D'),\!\,\diver z\in
H^{s-1}(D'), \, z\!\cdot\! n\in H^{s-1/2}(\Gamma')   \},
\end{align*}
where $n$ is the outward unit normal vector on~$\Gamma'$.
Moreover,   any function $z \in  H^s(D')$ satisfies the inequality
$$ \|z\|_s \le C \left( \|z  \| +\|\rot z\|_{s-1}+\|\diver z\|_{s-1}+\|  z \cdot n\|_{s-1/2}\right),$$
where $C$ depends only on $s$.
\end{lemma}
The proof of this lemma is given in the Appendix. Let us denote
$z= \nabla^\bot  u:=(\p_2u,-\p_1u)$. Then $\rot z= -\Delta u=-f$,
$\diver z =0 $.  Notice that (\ref{0:21121}) implies that  $z\cdot
n =0$. It follows from Lemma~\ref{L:rotdivihamar} and inequality
(\ref{E:poisdirichgn1121})
 that
$z\in  H^{s+1}(D')$ and $\| z\|_{ {s+1}  } \le C \|f\|_s $.
Thus, we obtain $u \in  H^{s+2}(D')$ and (\ref{E:poisdirichgn}).
\end{proof}
Let us take $g\in H^1(D')$ and consider the Neumann problem for
the Poisson equation:
\begin{align}
\Delta u=\diver g\,\,\, &\text{ in}\quad D',\label{0:11}\\
\frac{\p u}{\p n} =g \cdot n\,\,\,&\text{ on} \quad
\Gamma'.\label{0:21}
\end{align}
We say that $u \in \dot H^1(D') $ is a solution of  (\ref {0:11}),
(\ref {0:21})  if for any $\theta \in  H^1(D') $ we have
\begin{align*}
\int_{D'} \nabla u \nabla \theta \dd x =\int_{D'} g \nabla \theta \dd x.
\end{align*}
\begin{proposition}\label{T:poisMaz} For any    integer $s\ge1$  and  $ g\in
H^{s}(D')$ problem (\ref {0:11}), (\ref {0:21}) has a unique
solution $u\in\dot H^{s+1}(D')  $. Moreover,
\begin{align} \label{0:Pro2.3ign}
\| u\|_{\dot H^{s+1}} \le C \|g\|_s .
\end{align}
\end{proposition}
\begin{proof}
The Riesz representation theorem implies the existence of the
solution $u \in \dot   H^1(D') $. Lemma \ref{L:rotdivihamar}
applied to $z:=\nabla u$ gives (\ref{0:Pro2.3ign}).
\end{proof}

Now we consider the problem
\begin{align}
\Delta G_{a}=\p_1\delta_a\,\,\, &\text{ in}\quad \tilde D,\label{0:G1}\\
\frac{\p G_a}{\p n} = 0\,\,\,&\text{ on} \quad \p \tilde D,\label{0:G2}
\end{align}
where $\delta_a$ is the Dirac delta function concentrated at
$a=(a_1,a_2 )\in \tilde D   $.
\begin{proposition}\label{T:Green}
 Problem (\ref{0:G1}), (\ref{0:G2}) has a  solution
 $G_a~\!\in~\!C^\infty(\overline{\tilde D}~\!\setminus ~\! \{a\})$. Moreover, the following assertions hold:
 \begin{itemize}
   \item  [(i)]  For any open neighbourhood $Q$ of $a$ and for any integer $s\ge 1$, the
   solution $ G_{a}$ is uniquely determined by the additional condition that it belongs to $ \dot H^{s}(\tilde D\setminus \overline Q)$.
   \item   [(ii)]   For any $x\in \tilde D\setminus \{a\}$
 \begin{align}
 \nabla G_a(x )\!=\! -\!\frac{1}{2 \pi}
\left( \!\frac{|x-a|^2-2(x_1-a_1)^2}{|x-a|^4},\frac{-2(x_1-a_1)(x_2-a_2)}{|x-a|^4} \right)\!+\!
\psi_a(x) ,\label{0:grgnahat2}
\end{align}
where $\psi_a \in H^\infty(\tilde D)$.
  \item  [(iii)]Let $a\in \tilde D \setminus    \overline D $, then
    $G_a \in{\dot {H}}^\infty(   D  )$ and   for any integers $1~\!\le~\! i,j~\!\le~\!2 $  we have
 \begin{align}     \p_i\p_jG _a (x_1,x_2)\in\sS(D) .\label{0:grgnahat3}
\end{align}
 \item  [(iv)]  For any fixed $x\in\tilde D$ the function $ G_a (x)$  is analytic in $a\in \tilde D\setminus\{x
\}$.
    \end{itemize}
\end{proposition}\begin{proof} The existence of a solution $G_a~\!\in~\!C^\infty(\overline{\tilde D}~\!\setminus ~\! \{a\})$ will be established when proving assertion $(ii)$. To prove the uniqueness of the
solution, we assume that there are two solutions $G_{1,a}$ and
$G_{2,a}$. For $\tilde G=G_{1,a}-G_{2,a}$ we have
\begin{align*}
\Delta \tilde G=0\,\,\, &\text{ in}\quad \tilde D,\\
\frac{\p \tilde G}{\p n} = 0\,\,\,&\text{ on} \quad \p \tilde D,
\end{align*}
Let $\chi\in C_0^\infty(\tilde D)$ with $\chi =1$ in $Q$. Then
\begin{align*}
&\Delta (\chi \tilde G) =h,
\end{align*}
where $h\in C_0^\infty(\tilde D)$. The elliptic regularity for a
bounded domain implies that $\chi \tilde G\in H^\infty (\tilde
D)$. Since $\tilde G\in \dot H^s(\tilde D\setminus\overline Q)$,
we get $\tilde G \in \dot H^s (\tilde D)$. It follows from
Proposition \ref{T:poisMaz}  that $\tilde G=0$.

To prove    $(ii)$, we seek the solution in the form
\begin{align}\label{E.e.1ga}G_a= \p_1 (F_a\chi)+u_a,\end{align}
where $F_a(x)=-\frac{1}{2\pi} \ln|x-a|$ is the fundamental
solution of the Laplace operator in $\R^2$, $\chi\in C_0^\infty(
\tilde D)$, $\chi$ is $1$ in a neighborhood of $a$. Then $u_a$
must be the solution of the problem
\begin{align*}
\Delta u_{a}=-\p_1 (2\nabla F_a\cdot\nabla\chi+ F_a\Delta\chi):= \p_1 f\,\,\, &\text{ in}\quad  \tilde D, \\
\frac{\p u_a}{\p n}=0 \,\,\,&\text{ on} \quad \p \tilde D.
\end{align*}
Since $f\in C_0^\infty( \tilde D) $, applying  Proposition \ref{T:poisMaz}
 for $g=(f,0)$, we conclude that this problem has a   solution  $u_{a}\in ~\!H^{\infty }(\tilde D)$.
Property (\ref{0:grgnahat2}) follows from the construction of $G_a$.

Now let us show (\ref{0:grgnahat3}). We have that $G_a$ satisfies the
following problem in~$D$:
\begin{align}
\Delta G_{a}=0\,\,\, &\text{ in}\quad   D,\label{0:G112}\\
\frac{\p G_a}{\p n} = \ph\,\,\,&\text{ on} \quad \Gamma,\label{0:G212}
\end{align}
where $\ph \in C^\infty(\Gamma)$ and $\supp \ph\subset \overline \Gamma_0$. To show that the second derivatives
 of the solution belong to $\sS(D)$,
 let us  apply the Fourier transform in $x_1$ to  (\ref{0:G112}), (\ref{0:G212}). We obtain
\begin{align*}
 &\frac{d^2}{dx^2_2}\hat{G}_{a}-\xi^2\hat{G}_{a}  =0   \text{ in }    D, \\
&\frac{d \hat{G} _a}{dx_2}(\xi,-1) = \hat \ph_{1}(\xi),\\
&\frac{d \hat{G} _a}{dx_2}(\xi,1) = \hat \ph_2(\xi),
\end{align*}
where $\hat{G} _a$, $ \hat \ph_1$ and $\hat \ph_2$ are Fourier
transforms of $ {G} _a$, $   \ph(\cdot,-1)$ and $   \ph(\cdot,1)$, respectively. The solution
of this ODE is given by
\begin{align*}
\hat{G}_{a}(\xi,x_2) =  \frac{\hat \ph_2-\hat \ph_1}{2\xi\sinh(\xi)} \cosh(\xi x_2)+
\frac{\hat \ph_2+\hat \ph_1}{2\xi\cosh(\xi)} \sinh(\xi x_2).
\end{align*}
Since $\ph_1$ and $\ph_2$ are compactly supported, we have
 $$ \FF (   \p_i\p_jG _a ) \in \sS(D) ,  \quad 1\le i,j\le2,$$
 whence it follows that $ \p_i\p_jG _a \in S(D)$.
This completes the proof of $(iii)$.

Let  $\Omega $ be any   domain such that $  \overline\Omega\subset \tilde D$ and
 $\Omega \cap (\tilde D \setminus \overline D)\neq \emptyset$.
 Then
for any fixed $x\in\Omega$ the function $ G_a (x)$  is analytic in
$a\in \Omega\setminus\{x \}$. Indeed, let $\chi$ in
(\ref{E.e.1ga}) be  $1$ in $\Omega$. Then  the analyticity of $
G_a (x)$   is consequence of the facts that  $F_a$ is analytic in
$a$ and $u_a$ is a linear operator in $F_a$. Since $ G_a $ is the
unique solution of (\ref{0:G1}), (\ref{0:G2}), we have  the
analyticity of $ G_a (x)$ in $\tilde D\setminus \{ x\}$.
\end{proof}

 \subsection {Euler equations in an unbounded strip}
 We consider the
incompressible Euler system:
\begin{align}
\dot{u}+\langle u,\nabla\rangle u+\nabla p&=0,\quad  \diver u=0  \quad\text{in }
D,
\label{0:1}\\
u\cdot n&=0\quad \text{ on }  \Gamma  ,\label{0:3}\\
u(x,0)&=u_0(x),\label{0:4:init}
\end{align}
 It is well known that if $D$ is a bounded domain or if $D=\R^2$, then problem (\ref{0:1})-(\ref{0:4:init}) is
 well posed  in various function spaces (e.g., see  \cite{Kato2, KatoPonce, Wol}).

 In this subsection, we study the well-posedness of Euler system in $D$ defined by (\ref{0:cyl}).
\begin{definition} For any integer $s\ge3$ we say that $ (u,p)$ is a solution of Euler system if $(u,p) \in C(J_{T }, H^{s }(D))\times C
(J_{T },  \dot H^{s+1} (D)) $ and (\ref{0:1})   is satisfied in
the sense of distributions.
 \end{definition}

Let us show that the Euler system is  equivalent to the  problem
\begin{align}
&\,\dot{w}+\langle u,\nabla\rangle w =0, \quad w(x,0)= \rot
u_0(x),
\label{0:1ham}\\
&\rot u=w , \quad\diver u=0, \quad  u\cdot n|_\Gamma=0
\label{0:ham3} .
\end{align}
Clearly, if $(u,p) $ is a solution of the Euler system, then
(\ref{0:1ham}), (\ref{0:ham3})   hold. Now let us show that to any
solution $$(u,w) \in C(J_{T }, H^{s }(D))\cap C^1(J_{T }, H^{s-1
}(D))\times C (J_{T },   H^{s-1} (D)) $$   of (\ref{0:1ham}),
(\ref{0:ham3}) there corresponds a unique  solution
 $(u,p) \in C(J_{T }, H^{s }(D))\times C
(J_{T },  \dot H^{s+1} (D)) $  of  (\ref{0:1})-(\ref{0:4:init}). From  (\ref{0:1ham}) and (\ref{0:ham3}) it follows that
$$ \rot (\dot{u}+\langle u,\nabla\rangle u)=0.$$
Hence, there exists $p \in C (J_{T },   \dot H^{s} (D)) $ such
that $ -\nabla p = \dot{u}+\langle u,\nabla\rangle u$. It is easy
to see that
\begin{align*}
-\diver \nabla p&=\diver (\langle u,\nabla\rangle u)= \sum _{i,j=1}^2 \p_i u_j\p_j u_i \in H^{s-1} ,\,\, \rot \nabla p =0,\\
-\frac{\p p}{\p n}&= (\langle u,\nabla\rangle u) \cdot n=
 \langle u,\nabla\rangle   (u\cdot  \tilde n) -\sum _{i,j=1}^2    u_j u_i \p_j \tilde n_i \\&=-\sum _{i,j=1}^2    u_j u_i \p_j \tilde n_i \in H^{s-1/2},
\end{align*}
where $\tilde n$ is   a  regular extension of $n$. Thus, it
follows from Lemma \ref {L:rotdivihamar}  that $ \nabla p \!~\in
~\! C (J_{T }, H^{s}(D))$, whence we conclude that $p\in C (J_{T
},   \dot H^{s+1} (D))$ .

We have the  following result on the local well-posedness of Euler
system. The ideas used in the proof of existence of a solution
play an important role in the study of stabilization problem (see
Section \ref{S:2}). Therefore we present a rather complete  proof,
even though we do not really need this result.
\begin{theorem}\label{T:wp}
Let $s \ge 4$. For any $u_0\in H^s(D) $ satisfying the conditions
\begin{align*}
&\diver u_0 =0,\\
&u_0\cdot n =0   \text{ on }  \Gamma,
\end{align*}
there is   $T_*=T_*(\|u_0\|_s)$ such that system
(\ref{0:1})-(\ref{0:4:init}) has a unique solution $(u,p)\in
C(J_{T{_*}}, H^{s }(D))\times C (J_{T{_*}},   \dot H^{s+1} (D)) $.
\end{theorem}
\begin{proof}
\textbf{Uniqueness.} To prove the uniqueness, we argue  as in the
case of bounded domain. We assume that there are two solutions
$u_1$ and $u_2$. Then  for $v=u_1-u_2$, we have
\begin{align}
&\dot{v}+\langle u_1,\nabla\rangle v+\langle v,\nabla\rangle
u_2+\nabla p=0, \label{0:tarb1}\\
& \diver v=0
 , \quad  v\cdot n|_{\Gamma}=0   ,
 \quad
v(x,0)=0.\nonumber
\end{align}
Multiplying (\ref{0:tarb1}) by $v$ and integrating over $D$, we
get
\begin{align} \label{0:tarb2}
\p_t\|v(\cdot,t)\|^2\le-\int_D \langle u_1,\nabla\rangle v\cdot v
\dd x +C\|v(\cdot,t)\|^2-\int_D \nabla p \cdot v \dd x,
\end{align}
where $C  > 0$ is a constant depending only on $u_2$. Since
$u_1\cdot n=0$, the first term on the right-hand side of
(\ref{0:tarb2}) is zero. Let us show that the last term  is also
zero. Let us denote
 \begin{align*} \Omega_{(R)}:=\{ x\in D: \, |x_1|< R\},
\end{align*}
and let $\chi\in C ^\infty(\overline D)$ such that
\begin{align*}\chi(x)=
\left\{
  \begin{array}{ll}
    0, & \hbox{\text{if} $x\notin \overline\Omega_{(2)}$,} \\
    1, & \hbox{\text{if} $x\in\Omega_{(1)}$.}
  \end{array}
\right.
\end{align*} Clearly, we have
 \begin{align*}\lim_{R\ri\infty}\int_D \chi(\frac x R)\nabla p(x) \cdot v(x) \dd x=\int_D  \nabla p(x) \cdot v(x) \dd
 x.
\end{align*}
On the other hand, integrating by parts, we obtain
 \begin{align*} \int_D \chi(\frac x R)\nabla p(x) \cdot v(x) \dd x=-\int_{\Omega_{(2R)}\setminus\Omega_{(R)}}  \nabla\chi(\frac x R)  \frac{p(x)}{R} \cdot v(x) \dd
 x.
\end{align*}
Since $p\in \dot H^{s+1}$, from assertion $(iii)$ of Proposition
\ref{P:Hikaruc} we have $$\sup_{x\in \Omega_{(2R)}}
|\frac{p(x)}{R}|<C,$$ where $C$ does not depend on $R$.   Thus,
dominated convergence theorem yields
 \begin{align*}\int_D  \nabla p(x) \cdot v(x) \dd
 x=0.
\end{align*}
Applying the Gronwall inequality to (\ref{0:tarb2}), we obtain
$v=0$.

 \textbf{Existence.}
  To prove the existence of the
 solution, we shall need  the following result.
 \begin{lemma}\label{Lemma1}   Let $\tilde u\in C(\R_+, H^s)$,  $\tilde u \cdot n |_{\Gamma \times
\R_+}=0$, $f\in C(\R_+, H^{s})$ and $w_0\in~ H^{s}$,
$s\ge 3$. Then    the problem
\begin{align}
\p_t w  + \langle \tilde u ,\nabla\rangle w& =f,\label{0:transa}\\
w(x,0)&=w_0,\label{0:transb}
 \end{align}
 has a unique solution $w\in C(\R_{+}, H^{s})$, which satisfies
 the inequality
\begin{align}
\|w(\cdot, t)\|_{s} \le  \|w_0\|_{s}  +\int_0^t \left( \|f(\cdot,
\tau)\|_{s} + C\|w(\cdot, \tau)\|_{s}\|\nabla\tilde u (\cdot,
\tau)\|_{s-1}\right)\dd \tau .\label{00:lemma1ign}
 \end{align}
 \end{lemma}
 \begin{proof}
Let us denote by
$\phi^{g}:\tilde D\times\R_+\ri \tilde D$ the flow
associated to $g$, i.e., the solution of the problem
\begin{align*}
\frac{\p\phi^{ g} }{\p t}&=g(\phi^{g},t), \\
\phi^{  g}(x,0)&=x.
\end{align*}
Since (\ref{0:transa}), (\ref{0:transb}) is an inhomogeneous
transport equation, its solution is given by
$$w(\phi^{\tilde u}(x,t),t)=w_0(  x  )+\int_0^t f(\phi^{\tilde u}(x,\tau),\tau)\dd \tau. $$
Let us derive formally inequality
 (\ref{00:lemma1ign}). Taking the
$\p^\al:=\frac{\p^\al}{\p x^\al},$ $|\al|\le s$ derivative of
(\ref{0:transa})   and multiplying the resulting equation by
$\p^\al {w}$, we get
\begin{align*} \frac{1}{2}\frac{\dd}{\dd t}\|\p^\al{w}\| ^2  =&  \int_D \p^\al f\p^\al w \dd x
  - \int_D \p^\al\left( \tilde u \cdot\nabla\right)
{w} \cdot\p^\al {w}\dd  x
 \\
 \le&\left|\int_D(\tilde u \cdot\nabla) \p^\al{w}\cdot\p^\al {w}\dd
 x\right|+ \|f\|_{s} \|{w}\|_{s} +C\|\nabla{\tilde u}\|_{s-1} \|{w}\|_{s}^2.
  \end{align*}
Integrating by parts, one verifies that the first integral in the
right-hand side vanishes. Integrating in time, we obtain
(\ref{00:lemma1ign}).
 \end{proof}
 \begin{lemma}\label{L:curldiv} Let $w \in H^s, s\ge0$. Then the problem
  \begin{align}
{\rot z}&=w,\label{0E:dfdf}\\
\diver z&=0,\label{0E:dfdf2}\\
z\cdot n|_\Gamma&=0\label{0E:dfdf3}
\end{align}has a unique solution $z\in H^{s+1}$.
Moreover, there is $C>0$  depending only on  $s$ such that
  \begin{align}
\|z\|_{s+1 }&\le C \| w\| _{s }.\label{0E:dfdf4}
\end{align}
 \end{lemma}
\begin{proof}
Let us consider the following Dirichlet problem for the
Poisson equation:
\begin{align*}
\Delta v=  w\,\,\, &\text{ in}\quad D, \\
{v} =0 \,\,\,&\text{ on} \quad
\Gamma.
\end{align*} By Proposition \ref{T:poisMazDir}, $v\in H^{s+2}$ and $\|v\|_{s+2}\le C \| w\| _{s}$. Then
for $z=- \nabla ^\bot v$ properties
(\ref{0E:dfdf})-(\ref{0E:dfdf4}) are satisfied.
  \end{proof}

We now return to the proof of the theorem. The proof is based on
some ideas from  \cite{Bardfrisch} and \cite{BNSM}.

  \vspace{6pt}  \textbf{Step 1.} Let
$$E:H^k(D)\ri H^k(\R^2),\quad 0\le k\le s+1$$ be an extension
operator. Let $ \rho\in \sS(\R^2)$  be the function such that
\begin{align*}
\hat\rho(\xi)&=
\begin{cases}
\exp(-\frac{|\xi|^2}{1-|\xi|^2}) & \quad |\xi|<1, \\
0 & \quad |\xi|\ge1.
\end{cases}
\end{align*} Define $J_m:H^s(D)\ri
H^{s+1}(D)$ by
\begin{align}
J_m (v):=  (m^2\rho(m x)* E(v))|_D. \label{E:extoper}
\end{align}
For $u_0\in H^s(D)$ we define $ u^m_0 := J_m (u_0)$. Then
\begin{align}
&u^m_0\ri u_0 \text{ in } H^{s}(D),\quad \|u^m_0\|_s\le C\|u_0\|_s, \quad \|u^m_0 \|_{s+1}\le m C\|u_0\|_s , \label{E:2:regul1}\\
&\|u^m_0-u^k_0 \|_s=o(1)\,\text{ and }\|u^m_0-u^k_0
\|_1=o(\frac{1}{m^{s-1}})\text{ as } m\ri
\infty,\label{E:2:regul2}
\end{align}
where (\ref{E:2:regul2}) holds uniformly in   $k>m$. Using Lemmas
\ref{Lemma1} and \ref{L:curldiv}, we define the sequences $ u^m\in
C(\R_+,H^{s+1}) $ and  $ w^m\in C(\R_+,H^{s}) $  by
\begin{align*}
u^0= u_0,\\
\p_t w^{m+1} + \langle u^m,\nabla\rangle w^{m+1}=0, \quad w^{m+1}(0)=\rot u^{m+1}_0, \\
\curl u^{m+1}=w^{m+1}, \quad \diver u^{m+1}=0, \quad u^{m+1}\cdot
n|_{\Gamma }=0.
\end{align*}
 Our strategy is to show that sequence $u^m$ is convergent  and the
limit is the solution of Euler system. From (\ref{00:lemma1ign})
we derive
\begin{align}
\|w^{m}(\cdot, t)\|_{i} \le  \|\rot u^{m}_0\|_{i}  +C_1 \int_0^t
\| w^{m}(\cdot, \tau)\|_{i}\|   u^{m-1 }(\cdot, \tau)\|_{i} \dd
\tau\label{E:2.6-iapa152}
 \end{align}
for $i=s-1,$ $s$.

  \vspace{6pt}  \textbf{Step 2.} In this step, we show that  there exists a time $T_*=T_*(\|u_0\|_s)$ such
 that for any $t\in J_{T_{*}}$
\begin{align}
\|w^{m}(\cdot, t)\|_{s-1} \le  C \| u^{m}_0\|_s,\quad
\|w^{m}(\cdot, t)\|_{s} \le  C \| u^{m}_0\|_{s+1} \le mC
\|u_0\|_s. \label{E:2.6-iapa151}
 \end{align}
 By induction, let us prove for $i=s-1,s$ the inequality
\begin{align}
\|w^{m}(\cdot, t)\|_{i}\le  y_m(t),\label{E:2.6-indu}
 \end{align}
 where $C$ does not depend on $m$ and $y_m(t)$ is the solution of
\begin{align}
\dot y_m=C_1 y_m^2,\quad y_m(0)=\|\rot
u^{m}_0\|_{i}.\label{E:2.6-indu2}
 \end{align}
 Clearly (\ref{E:2.6-indu}) holds for $m=0$ for a sufficiently large $C$. Assume that it holds also
for $m-1$ and let us prove it for $m$. From the construction of $\hat
\rho$ we have $\| u^{m-1}_0\|_i\le \| u^{m}_0\|_i $, hence $
y_{m-1}\le y_m$.
 Thus, from (\ref{E:2.6-iapa152}), (\ref{E:2.6-indu2}) and induction hypothesis,  we have
\begin{align*}
\|w^{m}(\cdot, t)\|_{i} -y_m&\le C_1 \int_0^t (\| w^{m}(\cdot,
\tau)\|_{i}\|   u^{m-1 }(\cdot, \tau)\|_{i}-y_m^2) \dd \tau\\&\le
C_1 \int_0^t y_m(\| w^{m}(\cdot, \tau)\|_{i}-y_m) \dd \tau.
 \end{align*}
Inequality  (\ref{E:2.6-indu}) follows from the Gronwall
inequality. It is easy to see that  (\ref{E:2.6-indu}) yields
 (\ref{E:2.6-iapa151}).

  \vspace{6pt}  \textbf{Step 3.}
 Now let us show that   $ w^m$   converges in $C( J_{T{_*}},H^{s-1})$. In  view of Lemma \ref{L:curldiv},  sequence  $u^m$  converges in $C( J_{T{_*}},H^{s})$
and the limit $u$ is  the solution of Euler problem.

Notice that for $m<k$ we have
\begin{align}
\p_t\left(w^m-w^k\right)&+\langle  { u}^{k-1},\nabla\rangle
\left(w^m-w^k\right) =\langle   { u}^{k-1}-  {
u}^{m-1},\nabla\rangle w^m . \label{E:2.6-iapa}
\end{align}
Denote  $ K ^{m,k}(t  ):= \|w^m(\cdot, t  )-w^k(\cdot, t)
\|_{s-1 }$. Lemma \ref{Lemma1} implies
\begin{align}
  K ^{m,k}( t)  \le \|u^m_0-u^k_0 \|_s&+ C \int_0^ t    \big(  K ^{m,k}  ( \tau  ) \|   { u}^{k-1}(\cdot,
\tau )\|_{s -1}  \nonumber\\& +  \|u^{m-1}(\cdot, \tau   )-u^{k-1}(\cdot, \tau )
\|_{s-1 }\|w^m( \cdot, \tau)\|_{s}          \big)\dd \tau.\label{E:Kigna122}
\end{align}
On the other hand,
\begin{align}
&\|w^m \|_{s} \le Cm, \quad \|u^{m-1} -u^{k-1} \|_{s-1 } \le
\|u^{m-1} -u^{k-1} \|_1^{\frac{1}{s-1}}\|u^{m-1} -u^{k-1}
\|^{\frac{s-2}{s-1} }_{s}  .\label{E:Kigna12298}
\end{align}
Assume for a moment that
 \begin{align}
U^{m,k}:=\|w^{m-1} -w^{k-1} \|\le o(\frac{1}{m^{s-1}})
\label{E:Kigna12454sd2}.\end{align} Substituting
(\ref{E:Kigna12298}) into (\ref{E:Kigna122}) and using
(\ref{E:2:regul2}) and  (\ref{E:Kigna12454sd2}), we obtain
\begin{align*}
  K ^{m,k}( t) \le o(1)+ C \int_0^ t    \big(  K ^{m,k}  ( \tau  ) \|   { u}^{k-1}(\cdot,
\tau )\|_{s -1} \big)\dd \tau.
\end{align*}
Using the Gronwall inequality, we obtain the convergence of
  $w^m$   in $C(J_{T{_*}}, H^{s -1}(D))$.

  \vspace{6pt}  \textbf{Step 4.}
To complete the proof of the theorem, it remains to show
(\ref{E:Kigna12454sd2}). Taking the scalar product of
(\ref{E:2.6-iapa}) with $w^m-w^k$ in $L^2$, we get
\begin{align*}
  U ^{m,k}( t)  \le C \|u^m_0-u^k_0 \|_1+ C \int_0^ t     &U ^{m-1,k-1}  ( t_1  )    \dd t_1.
\end{align*}
Iterating this inequality, one deduces
\begin{align}
  U ^{m+p,k+p}( t)   \le& C\|u^{m+p}_0-u^{k+p}_0 \|_1+ C \int_0^ t    U ^{m+p-1,k+p-1}  ( t_1 )    \dd t_1 \nonumber
\\ \le& C \|u^{m+p}_0-u^{k+p}_0 \|_1+ C \int_0^ t  \big(  C \|u^{m+p-1}_0-u^{k+p-1}_0 \|_1\nonumber
\\& + C \int_0^ {t_1}    U ^{m+p-2,k+p-2}  ( t_2 ) \big)   \dd t_2   \dd t_1 \nonumber
\\ \le&   C\|u^{m+p}_0-u^{k+p}_0 \|_1+ C \int_0^ t \big(C \|u^{m+p-1}_0-u^{k+p-1}_0\| _1 \nonumber
\\&  +\cdots+
C \int_0^{ t_{p-1}}  \big(C \|u^{m+1}_0-u^{k+1}_0 \|_1+ C \int_0^
{t_p} U ^{m,k}  ( t_p )   \big )\big)\dd t_p \cdots  \dd t_2   \dd
t_1.\nonumber
\end{align}
Hence, for any $t\in J_{T_*}$ we obtain
\begin{align}
  U ^{m+p,k+p}& \le \sum_{j=1}^p  \frac{C^{p-j+1} {T_*}^{p-j}}{(p-j)! } \|u^{m+j}_0-u^{k+j}_0 \|_1+
\frac{C^{p +1} {T_*}^{p }}{{p }! } \max_{t\in
  [0,T_*]}U^{m,k} .\label{E:Kidfgnzdfa122}
\end{align}
Since  $$ \sum_{j=1}^\infty  \frac{C^{j+1} {T_*}^{j}}{j! } <\infty,$$
inequalities (\ref{E:2:regul2}) and (\ref{E:Kidfgnzdfa122})    imply (\ref{E:Kigna12454sd2}).

\end{proof}
\begin{remark}\label{rem:2.6} We have the following assertions:
 \begin{itemize}
                                \item  Adapting the Beale--Kato--Majda criterion
  (see \cite{BKM}) for an unbounded strip, one can prove that the solution of (\ref{0:1})-(\ref{0:4:init}) is global in
  time. However, we shall not need this result.
                                \item  Let us take any non-zero function $g\in H^{s-1/2}(\Gamma)$.
If   the homogeneous boundary condition (\ref{0:3}) is replaced by
$u\cdot n|_{\Gamma  }=g$, then, to our knowledge, neither existence
nor uniqueness of a solution is known to hold (even in the case of
   bounded domain).
                              \end{itemize}

\end{remark}

 \section{Main result} \label{S:2}
Let $D$  and $\Gamma_0$ be defined by (\ref{0:cyl}) and (\ref{0:gammaisah1}). Consider the Euler system:
\begin{align}
\dot{u}+\langle u,\nabla\rangle u+\nabla p=0&\qquad \text{in}\quad
D\times (0,\infty),\label{1:1}\\ \diver u=0
 &,
\label{1:2}\\
u\cdot n=0&\qquad \text{on}\quad \Gamma\setminus\Gamma_0\times \R_+,\label{1:3}\\
u(x,0)=u&_0(x).\label{1:4:init}
\end{align} For any integer $s$ we denote
$$ \XX^s(D)= C(\R_+, C_b (\overline D)\cap \dot H^s(D) ),  $$
and $ \langle
x_1\rangle:=(1+x_1^2)^{1/2}$.  The following theorem is our main result.
\begin{theorem}\label{T:himnth} For any constants $\al,\beta>0$, $c\in \R$ and integer $s\ge 4$, for any initial data $u_0\in H^s(D)$  such that
\begin{align}
&\diver u_0 =0,\label{1:divpaym}\\
&u_0\cdot n =0   \text{ on }  \Gamma\setminus\Gamma_0\label{1:ezrpaym},\\
&\| \exp (\al\langle x_1\rangle^{2+\beta})  \rot u_0(x_1,x_2) \|_{s-1 }<\infty\label{1:expdecay}
\end{align}
 there is a solution $(u,p) \in
\XX^{s }(D) \times C  (\R_+, \dot H^{s}(D)) $  of
(\ref{1:1})-(\ref{1:4:init})  with
\begin{align}\label{1:karavardy}
 \lim_{t \rightarrow \infty }( \|u(\cdot,t)-(c,0)\|_{L^\infty(D)}+\| \nabla u(\cdot,t)\|_{s-1 } +\| \nabla p(\cdot,t)\|_{s-1 } )=0.
 \end{align}
\end{theorem}
As explained in Introduction, in this formulation  the control is
not given explicitly, but we can assume that control acts on the
system as a boundary condition on $\Gamma_0$.  So we show that
there exists control $\eta$  such that there is a solution of our
system with $ u\cdot n|_{\Gamma_0}=\eta$ verifying
(\ref{1:karavardy}). As we mentioned in Remark \ref{rem:2.6}, we
are not able to show that this solution
 is unique.

Using a standard scaling argument for Euler system, we can reduce
this theorem to a small neighborhood of the origin.
\begin{theorem}\label{T:prop} There exists $\e>0$ such that for any  $u_0\in H^s(D)$ and $c\in \R$  verifying (\ref{1:divpaym})-(\ref{1:expdecay}) and
\begin{align*}
\|u_0\|_{s }&< \e,\quad |c|<\e
\end{align*}
there is a solution $(u,p) \in \XX^{s}(D)\times C (\R_+, \dot H^{s
 } (D) )$ of   (\ref{1:1})-(\ref{1:4:init}) satisfying\!~
 (\ref{1:karavardy}).
\end{theorem}
\begin{proof}[Proof of Theorem \ref{T:himnth}]
Let $\e>0$ be the constant in Theorem \ref{T:prop}. Take any
$u_0\in H^s  (D)$ and $c\in \R$  verifying
(\ref{1:divpaym})-(\ref{1:expdecay}).  Let   $M>0$ be  such that
\begin{align*}
\bigg\|\frac{u_0}{M}\bigg\|_{s }&< \e,\quad
\bigg|\frac{c}{M}\bigg|<\e.
\end{align*}
By Theorem \ref{T:prop}, there exists a solution $(u_M,p_M)$
of (\ref{1:1})-(\ref{1:3}) with initial condition
$u_M(0)=\frac{u_0}{M}$,  such that \begin{align*}
 \lim_{t \rightarrow \infty }( \|u_M(\cdot,t)-(\frac{c}{M},0)\|_{L^\infty(D)}+\| \nabla u_M(\cdot,t)\|_{s-1 }+\| \nabla p_M(\cdot,t)\|_{s-1 } )=0.
 \end{align*}
Then $(u,p)=(Mu_M(x,Mt),M^2p_M(x,Mt))$ is a solution of our system
with $u (0)= u_0 $ and it satisfies (\ref{1:karavardy}).
\end{proof}

\begin{proof}[Proof of Theorem \ref{T:prop}]
The proof of this theorem is based on generalization of the Coron
return method to the case of an unbounded strip. It consists in
construction of a particular solution
 $(\overline u,\overline p)$ of (\ref{1:1})-(\ref{1:3}) such that the solution of linearized
  system around $(\overline u, \overline p)$ verifies  property (\ref{1:karavardy}). Then, in the
  small neighborhood of $\overline u$,  we construct a solution $u$  of Euler system  satisfying (\ref{1:karavardy}).

  \vspace{6pt}  \textbf{Step 1.} In this step, we construct a particular solution $(\overline u,\overline p)$ of  (\ref{1:1})-(\ref{1:3})
   such that any point of strip $D$, driven by the flow of $\overline u $, leaves $\overline D$ at some time.
Let $\hat D \subset \R^2$ be the strip
 \begin{align*}
\hat{D} :=\{(x_1,x_2): \, x_1\in \R, x_2\in(-2,2)\}.
\end{align*}
 Let us admit the  proposition below, which is proved
in Section \ref{S:5}.
\begin{proposition}\label{P:Tetakar} There are scalar functions $ \theta^i \in C^1( \hat D  \times\R_+)$ with
  $\nabla \theta^i\in~\!\! \XX^{s}(\hat D  )$, open balls $B^i$, a sequence $\tau_i\subset \R_+$,
  constants $M,\lambda$ and an integer $N\in \N$ such that the following properties are true.
 \begin{enumerate}
   \item  \underline{Covering.} For any integer $k\ge 0$, we have
   \begin{align}
[k,k+1]\times[-1,1]&\subset\bigcup_{j=1}^N B^{2kN +j},  \label{E:3 gnd1}\\
[-k-1,-k]\times[-1,1]&\subset\bigcup_{j=1}^N B^{(2k+1)N+j}.\label{E:3 gnd2}
\end{align}
In particular, the union of  balls
$B^i$  covers $\overline D$ and any square $[k,k+1]\times[-1,1]$ is
covered by $N$ balls.
   \item \underline{Support.}  \begin{align}
\supp \theta^i \subset   \hat D \times (0,\tau_i).\label{1:Thetakrich}
\end{align}
   \item \underline{Vector field.} The time dependent vector field $\nabla \theta^i $ is divergence-free
in $D$
   and tangent to $\Gamma\setminus \Gamma_0$ and $\p \hat D$:
 \begin{align}
&\Delta\theta^i=0 \qquad\text{in}\quad D\times [0,\tau_i],\label{1:11}\\
&\frac{\p\theta^i}{\p n} =0\qquad \text{on}\quad  (\Gamma\setminus
\Gamma_0 )\cup \p\hat D \times [0,\tau_i] .\label{1:12}
\end{align}
       \item \underline{Time decay.} For any $i\ge 1$ we have
              \begin{align}
||\nabla\theta^i(\cdot,t)||_{\XX^s(\hat  D)} &\le\frac{1}{i}\text { for any }
t\in[0,\tau_i],\label{1:thetaverj2}\\
\tau_i&\le M i.\label{1:tauignah}
\end{align}
 \item \underline{Flow.}  For any $i\ge 1$ and $c\in \R$ with $|c|<\lambda$ the flow associated with $\nabla\theta^i+(c,0)$ is such that
\begin{align}
&\phi^{\nabla\theta^i+(c,0)}(B^i,\tau_i)\subset \hat D \setminus \overline D
\label{1:thetaverj1}.
\end{align}
Moreover, there are two closed balls $\tilde B_1, \tilde B_2
\subset \hat D \setminus \overline D$ such that
\begin{align}
& {\cup_{i=1}^\infty
\phi^{\nabla\theta^i+(c,0)}(B^i,\tau_i)}\subset \tilde B_1\cup
\tilde B_2\label{1:thetaverj1ar22}.
\end{align}
 \end{enumerate}
\end{proposition}
Let us set $t_0=0$,
\begin{align}\label{E:tierikar}
t_i=2\sum_{j=1}^i\tau_j,\quad t_{i+1/2}=\frac{t_i+t_{i+1}}{2},
\quad i\ge1.
\end{align}
We define $\overline\theta$ in the following way:
\begin{align}
\overline\theta(x,t)&=\theta^i(x,t-t_{i-1})  \text { for } t\in [t_{i-1},t_{i-1/2}],\label{E:Tgcikar1}\\
\overline\theta(x,t)&=-\theta^i(x,t_{i}-t) \text { for } t\in
[t_{i-1/2},t_{i}].\label{E:Tgcikar2}
\end{align}
Notice that  from the construction of $t_i$ we have
$t_i-t_{i-1/2}=\tau_i$. Thus (\ref{1:Thetakrich}) shows that
$\overline \theta  \in C^1( \hat D\times\R_+)$ and  $\nabla \overline \theta\in~\!\! \XX^{s}(\hat D)$.  We define
\begin{align*}
\overline{u}:&=\nabla\overline\theta +(c,0),\\
\overline{p}:&=-\p_t\overline\theta-\frac{|\nabla\overline\theta|^2}{2}-c\p_1\theta.
\end{align*}
Then  $(\overline{u},\overline{p} )$  is a  solution of
(\ref{1:1})-(\ref{1:3}). Indeed, by construction,  $
(\overline{u},\overline{p} ) $ satisfies   (\ref{1:1}). Properties
(\ref{1:11}) and (\ref{1:12}) imply (\ref{1:2}) and
(\ref{1:3}), respectively. Moreover, it follows from
(\ref{1:thetaverj2}), (\ref{1:thetaverj1}) that for any $i\in\N$,
we have
\begin{align}
&\,\,\,\phi^{\overline{u}}(B^i,t_{i-1/2})\not\subset \overline D,\nonumber\\
&\lim_{t \rightarrow \infty }( \|\overline
u(\cdot,t)-(c,0)\|_{L^\infty(D)}+\| \nabla \overline
u(\cdot,t)\|_{s-1 } )=0.\label{E:sbalizard}
\end{align}

We deduce  from (\ref{E:Tgcikar1}) and (\ref{E:Tgcikar2}) that
\begin{align}\label{E:tipahimasin}
\phi^{\overline{u}}(x,t_{i})&=x
\end{align}
for any $i\ge 1$ and  $x\in  \hat D$.
We shall need the following result, which is proved in Section \ref{S:5}.
\begin{proposition}\label{L:flow}  There is a constant $\nu>0$ such that the functions $ \theta^i  $ in  Proposition \ref{P:Tetakar} can be chosen
  in a way that, for any $u\in  \XX^{s}( \hat D)$ satisfying the inequality  $$ \int_0^\infty\|u(t)-\overline{u}(t)\|_{s,\hat D}\dd t\le \nu,$$  we have
$\phi^{u}(B^i,t_{i-1/2})\subset \hat D\setminus \overline D$ for any $i\ge
1$.
\end{proposition}
 From now on, we assume that functions $\theta^i$ verify this proposition.

  \textbf{Step 2.} In this step, we construct an application $F_{u_{0}}$  such that its   fixed point
    is a solution of our stabilization problem.
First, for any constant $\nu>0$ let us introduce the set
\begin{align*} \YY_\nu(u_0):=\{&u\in  \XX^s(D): \quad\diver u=0,
 \quad \int_0^\infty\|u(t)-\overline{u}(t)\|_{s,D}\dd t\le \nu,
 \\&u(x,t)\cdot n(x)
 = (u_0(x)\mu(t)+\overline u(x,t) )\cdot n(x)  \text{ on } \Gamma\times \R_+\},\end{align*}
where $\mu   \in C_0^\infty  ([0,\infty  )  )$  is a non-negative
function such that
\begin{align*}
\mu(0)=1,\quad \int_0^\infty \mu(t)\dd t<1 .\end{align*}
 Let $D_1:= \R \times (-\frac{3}{2},\frac{3}{2})$
 and  $\pi: H^s(D) \rightarrow  H^s(\hat D) $ be any linear bounded extension
operator such that $\supp \pi u \subset D_1$ for any $u\in H^s(D) $. Let $\kappa^i\in C_0^\infty(\hat D)$ be a
partition of unity subordinate to $B^i$, i.e.,
\begin{align*}
\supp\kappa^i&\subset B^i,\\
\sum_{i=1}^\infty\kappa^i&=1\text{ in } \overline D.
\end{align*}
Take any $u\in \YY_\nu(u_0)$ and let $w^l\in C  (\R_+,H^{s-1}  (
\hat  D))$ be the solution of the   linear problem
\begin{align} \label{E:wlihavas}
&\dot{w}^l+\langle \tilde {u},\nabla\rangle w^l =-(\diver \tilde {u}) w^l\text{ in }\hat {D}\times \R_+,\\
&w^l(0)=\kappa^l\rot( \pi u_0),\label{E:wlihavas12}
\end{align}
where
\begin{align}\label{E: uisharunak}
\tilde u=\overline u+\pi( u-\overline u).
\end{align}
Take  $\nu$   such that Proposition  \ref{L:flow} holds. Since $\supp
w^l(0)\subset B_l$, we obtain
 \begin{align} \label{E:wl@tlum}
{w}^l(x,t_{l-1/2})=0 \text{ for any } x\in \overline D.
\end{align}
For any $t\in\R_+$ we define the   function
 \begin{align}\label{1:wisahm}
  w(\cdot,t)&=\sum_{l=i+1}^\infty w^l(\cdot,t),\quad\text{ when }
  t\in[t_{i-1/2},t_{i+1/2}],
  \end{align}
  where $t_{-1/2}:=0$ and $i\ge 0$. Let us show that for any $t\in[t_{i-1/2},t_{i+1/2}]$ the sum in
the right-hand side of (\ref{1:wisahm}) exists  and   belongs to
$C  (\R_+,H^{s-1}  (  D)  )$. Applying  Lemma \ref {Lemma1} to
(\ref{E:wlihavas}), (\ref{E:wlihavas12}), we obtain
 \begin{align*}
\| w^l(t)\|_{s-1,\hat D}\le C( \| \kappa^l\rot( \pi u_0)
\|_{s-1,\hat D}+ \int_0^t\|\nabla\tilde u (\tau  )\|_{s-1,\hat
D}\| w^l (\tau  )\|_{s-1,\hat D}\dd \tau   ).
  \end{align*}
It follows from the Gronwall inequality and   relation (\ref{E:
uisharunak}) that
 \begin{align*}
\| w^l&\left(t\right)\|_{s-1,\hat D}\le C  \| \kappa^l\rot\left(
\pi u_0\right)  \|_{s-1, \hat D}\exp  \left(C
\int_0^t\|\nabla\tilde u\left( \tau\right)\|_{s-1,\hat D} \dd \tau
\right)
\\ &\le C \| \kappa^l\rot\left( \pi u_0\right) \|_{s-1,\hat
D}\exp\left(C \int_0^t\left(\|\nabla \overline u
\left(\tau\right)\|_{s-1,\hat D}+\|\overline u \left(\tau \right)
-\tilde u   \left(\tau \right)\|_{s,\hat D}\right) \dd \tau
  \right) .
  \end{align*}
  Using  the fact that $\overline u \in \XX^{s}(\hat D)$,  we get
 \begin{align*}
\| w^l  (t)\|_{s-1,\hat D}&\le   C  \| \kappa^l \rot  ( \pi u_0 )
\|_{s-1,\hat D}
 \exp( C( t_{i+1/2}+\nu))
   \end{align*}
 for any
  $t\in[t_{i-1/2},t_{i+1/2}]$. Thus
 \begin{align}\label{E:wlignahata} \sum_{l=i}^\infty \|
w^l(t)\|_{s-1,\hat D}&\le   C \exp  (Ct_{i+1/2})
\sum_{l=i}^\infty  \| \kappa^l\rot ( \pi u_0  ) \|_{s-1,\hat D}.
   \end{align}
  Using
  (\ref{1:expdecay}) and assertion 1 of Proposition \ref{P:Tetakar}, we derive that the right-hand side   of
(\ref {E:wlignahata}) is finite. Hence, $w \in~\!C
([t_{i-1/2},t_{i+1/2}],H^{s-1}  (\hat D))$ for any $i\ge 0 $.
Moreover, assertion    (\ref{E:wl@tlum}) yields that  $w$ is
continuous at $t_{i-1/2}$, thus $w\in C(\R_+,H^{s-1}(D)  )$ (we
emphasize that, in general,  this is not true for $\hat D$).
Furthermore, we have
\begin{align*}
\dot{w} +\langle \tilde {u},\nabla\rangle w =-(\diver \tilde {u}) w \quad\text{ in }&\quad {\hat D}\times [t_{i-1/2},t_{i+1/2}],\\
w(0)= \sum_{l=1}^\infty \kappa^l \rot \pi u_0\qquad \text{ in
}&\quad {\hat D}.
\end{align*}  In Step 3, we prove  that for this $w$ there exists a
  $v\in \YY_\nu(u_0)$
  such that
\begin{align}
\rot v={w}.\label{1:rotih}
\end{align}
For any $u\in \YY_\nu(u_0)$, let $F_{u_0}(u):=v$. In Step 4, we
show that the mapping $F_{u_0}:\YY_\nu(u_0)\ri\YY_\nu(u_0)$ has a
fixed point. We shall prove that this fixed point is a solution of
our stabilization problem.

  \vspace{6pt}  \textbf{Step 3.} In this step, we prove the
  existence of the solution $v\in
\YY_\nu(u_0)$ of (\ref{1:rotih}). By Lemma \ref{L:curldiv},  there
is a function $z\in C(\R_+,H^s( D ))$ such that
 \begin{align}
{\rot z}&=w,\nonumber\\
\diver z&=0,\nonumber\\
z\cdot n&=0,\nonumber\\
\|z(\cdot,t)\|_{s,D}&\le C \|w(\cdot,t)\|_{s-1,D}.\label{E:dfdf3}
\end{align}
 Let us take the solution of the following problem
\begin{align*}
\Delta \ph&=0   \text{ in } {D},  \\
\frac{\p\ph}{\p n}  &=(u_0 \mu   ) \cdot n   \quad \text{on}\quad
\Gamma.
\end{align*}
From Proposition \ref{T:poisMaz} we have $\ph\in  C(\R_+,\dot
H^{s+1}(D))$ and  \begin{align*}  \|\ph(\cdot,t) \|_{\dot
H^{s+1}(D)}\le C\| u_0 \mu(t) \|_{s,D}.
\end{align*}
Denote   $v=z+\nabla\ph+\overline u $. Let us show that $v\in
\YY_\nu(u_0)$ and (\ref{1:rotih}) is verified. Clearly
\begin{align*}
\rot v&=\rot z=w,\\
 \diver v&=\diver z+\Delta \ph=0, \\
 v \cdot n  &=( u_0(x)\mu +\overline u  )\cdot n  \text{ on }
\Gamma\times \R_+.
  \end{align*}
Hence, to show $v\in \YY_\nu(u_0)$, it suffices to prove
 for
sufficiently small  ${u_0}$ that
\begin{align}\label{1:uhany}
\int_0^\infty\|v(t)-\overline{u}(t)\|_{s,  D} \dd t\le \nu.
\end{align}
   It follows from the construction of $v$ that $$ \|v(\cdot,t)-\overline u(t)\|_{s,D}\le \|\ph(\cdot,t) \|_{\dot H^{s+1}(D)} +\|z(\cdot,t)\|_{s,D} .$$
  Proposition \ref{T:poisMaz} and (\ref {E:dfdf3}) imply
$$\int_0^\infty\|v(t)-\overline{u}(t)\|_{s,D} \dd
t\le \|u_0\|_{s,D} \int_0^\infty\mu(t) \dd t+
C\int_0^\infty\|w(\cdot,t)\|_{s-1,D} \dd t.$$
 From
(\ref{1:wisahm}) we have
\begin{align*}
\int_0^\infty\|w(\cdot,t)\|_{s-1,D} \dd
t&=\sum_{i=0}^\infty\int_{t_i-1/2}^{t_{i+1/2}} \big\|
\sum_{l=i+1}^\infty  w^l(\cdot,t) \big\|_{s-1,D} \dd t.
\end{align*}
Applying Lemma \ref{Lemma1} to $\sum_{l=i+1}^\infty  w^l$, we
obtain
\begin{align*}
 \!\big\|\sum_{l=i+1}^\infty\!  w^l\left(x,t\right) \big\|_{s-1,D}  \le
 C\exp \left(C\int _0^t \|\nabla\tilde u\left(\!\cdot,\tau\right)\|_{s-1,D} \dd \tau\!\right)\!\big\|\sum_{l=i+1}^\infty
\!\kappa^l \!\rot  u_0 \big\|_{s-1,D}.
\end{align*}
Thus
\begin{align}
\int_0^\infty\|w\left(\cdot,t\right)\|_{s-1,D} \dd t&\le C
\sum_{i=0}^\infty\int_{t_i-1/2}^{t_{i+1/2}}\big\|\sum_{l=i+1}^\infty
\kappa^l \rot u_0 \big\|_{s-1,D}\times\nonumber\\
& \quad\times \exp \left(C\int
_0^t \|\nabla\tilde u\left(\cdot,\tau\right)\|_{s-1,D} \dd \tau\right) \dd t\nonumber\\
&\le C_1 \sum_{i=0}^\infty\int_{t_i-1/2}^{t_{i+1/2}} \exp
\left(Ct_{i+1/2}\right)
  \|\rot u_0 \|_{s-1, D \backslash
\cup_{l=1}^{i}B_l}\dd t.\nonumber
\end{align}
Combining (\ref{1:expdecay}),   (\ref{1:tauignah}),
(\ref{E:tierikar}) and assertion 1 of Proposition \ref{P:Tetakar},
we get
$$ (t_{i+1/2}- t_{i-1/2})\exp (Ct_{i+1/2}) \| \rot u_0 \|_{s-1, D \backslash \cup_{l=1}^{i}B_l}\le C_2\frac{1}{i^2}$$
  for any $i>0,$ where $C_2$ does not depend on $i$.   Let
  $K$ be a  constant
such that
\begin{align*}
C_1 C_2\sum_{i=K}^\infty \frac{1}{i^2}<\frac{\nu}{2}.
 \end{align*}
Taking $u_0$ sufficiently small  such that
\begin{align*}
\|u_0\|_{s,D} & +
\sum_{i=1}^{K}\int_{t_i-1/2}^{t_{i+1/2}}\sum_{l=i+1}^\infty
\|\kappa^l\rot u_0\|_{s-1,D}  \exp \left(\int _0^t \|\nabla\tilde
u\left(\cdot,\tau\right)\|_{s-1 ,D} \dd \tau\right)\dd t \\& \le
\frac{\nu}{2},
 \end{align*}
we get (\ref{1:uhany}).

\vspace{6pt}  \textbf{Step 4.} In this step, we show that the
mapping $F_{u_0}:\YY_\nu(u_0)\ri\YY_\nu(u_0)$ admits a fixed point, which is the
solution of our stabilization problem.  Let us take a sequence $u^m_0 := J_m(u_0)$, where $J_m$ is the operator defined by (\ref{E:extoper}).
 We have that $ u^m_0 \in H^{s+1}(D)$   verifies (\ref{E:2:regul1}), (\ref{E:2:regul2}).
Take $u^0(x,t)=\mu(t  )u_0 (x)+\overline  u (x,t ) $. For
sufficiently small $u_0$ we have $u^0\in \YY_\nu(u_0)$. Let
$u^1=F_{u^1_0} (u^0 )$  and let $w_1$ be defined as in
(\ref{1:wisahm}) with $u=u^0$ and $u_0 (x)=u^1_0 (x)$. In this way
we introduce the sequences $u^m\in \XX^{s }$  and $w_m\in
C(\R_+,H^{s}(  D))$ by the relations
$$
 \left\{
\begin{array}{ll}
        u^{m+1}=F_{u^{m+1}_0} (u^m) , \\
     w_{m+1}  \text { defined as in   (\ref{1:wisahm}) with } u= u^m \text{ and } u_0=u^{m+1}_0. \\
\end{array}
\right .$$ Let us show the convergence of
 $w_{m}$ in $C([0,t_{1/2}],H^{s-1}(\hat D))$. This will be proved by
  using the same arguments  as in the proof of Theorem \ref{T:wp}.   It is easy to see
\begin{align*}
\p_t\left(w_m-w_k\right)&+\langle \tilde { u}^{k-1},\nabla\rangle
\left(w_m-w_k\right) =\langle  \tilde { u}^{k-1}- \tilde {
u}^{m-1},\nabla\rangle w_m   \\&- \left(\diver \tilde { u}^{k-1}
\right)  \left(w_m-w_k  \right) - \left(\diver { \tilde u}^{m-1}-
\diver \tilde { u}^{k-1} \right)w_m.
\end{align*}
Setting  $ K ^{m,k}(t  ):= \|w_m(\cdot, t  )-w_k(\cdot, t)
\|_{s-1,\hat D}$ and using Lemmas \ref{Lemma1} and
\ref{L:curldiv}, we obtain
\begin{align}
  &K ^{m,k}( t)  \le \|u^m_0-u^k_0 \|_s+ C \int_0^ t   \big(  K ^{m,k}  ( \tau  ) \| \nabla \tilde  { u}^{k-1}(\cdot,
\tau )\|_{s-1 }  \nonumber\\& +  \|\tilde u^{m-1}(\cdot, \tau
)-\tilde u^{k-1}(\cdot, \tau ) \|_{s-1 }\|w_m( \cdot, \tau)\|_{s}
+  K ^{m-1,k-1}  ( \tau  )\|w_m( \cdot, \tau)\|_{s-1}     \big)\dd
\tau.\label{E:Kigna}
\end{align}
Let us show that for any $m\in\N$
\begin{align}\label{E: lucmansahman}
 \sup_{t\in [0,t_{1/2}]}\|w_m(\cdot, t)\|_{s-1,\hat D}< C\|u^m_0\|_{s,\hat D},
\end{align}
where $C$ depends only on   $  \|\overline u  (t)\|_{L^1 (
(0,t_{1/2}),\dot H^s  ( \hat D  )  )}$  and does not depend
on $m$. From the construction of $ w_m$, we have
\begin{align*}
\dot{w}_m +\langle \tilde {u}^{m-1},\nabla\rangle w_m = -(\diver
{ \tilde u}^{m-1}) {w_m} \qquad &\text{in}\quad {\hat D}\times \R_+,\\
w_{m}(0)= \sum_{l=1}^\infty \kappa^l \rot \pi u^m_0\qquad
&\text{in}\quad {\hat D}.
\end{align*}
Applying Lemma \ref{Lemma1}, we get
\begin{align*}
\|{w_m}  &\left(t\right)\|_{s-1,\hat D} \le C  \left(
\|u^m_0\|_{s,\hat D}+ \int_0^t \|{w_m   }\|_{s-1,\hat D} \|\nabla \tilde
{u}^{m-1} \|_{s-1,\hat D} \dd t\right) \\& \le C \left(
\|u^m_0\|_{s,\hat D}+
  \int_0^t \|{w_m}\|_{s-1,\hat D}   \left(\|\nabla \overline {u} \|_{s-1,\hat D}
+\|  \overline{u} -\tilde {u}^{m-1} \|_{s,\hat D} \right)\dd t
\right).
\end{align*}
Using the Gronwall  inequality and the fact that $\tilde {u}_{m-1}
\in \YY_\nu( u^m_0 )$, we derive
\begin{align*}
\|{w_m}  (t  )\|_{s-1,\hat D} &\le C  ( \|u_0\|_s\exp
(\int_0^{t_{1/2}}   (\|\nabla  \overline {u} \|_{s-1,\hat D}+\|  \overline{u}
-\tilde {u}_{m-1} \|_{s,\hat D} )\dd t ) \le C_1,
\end{align*}
where $C_1$ does not depend on $m$.  Thus, we obtain (\ref{E:
lucmansahman}). The construction of $u^m$ implies  boundedness of $\sup _{t\in
[0,t_{1/2}]}\|u^m\|_{s,\hat D}$ uniformly in $m$. In the same way  we can show that
\begin{align*}
 \sup_{t\in [0,t_{1/2}]}\|w_m(\cdot, t)\|_{s,\hat D}\le C\|u^m_0\|_{s+1,\hat D}.
\end{align*}
Combining this with (\ref{E:2:regul1}) and (\ref{E:2:regul2}), we get
\begin{align}   \|\tilde u^{m-1}(\cdot, \tau)-\tilde u^{k-1}(\cdot, \tau )
\|_{s-1 }&\|w_m( \cdot, \tau)\|_{s} \le  \|\tilde u^{m-1}(\cdot,
\tau)-\tilde u^{k-1}(\cdot, \tau ) \|^{1/s} \times \nonumber\\
&\times\|\tilde u^{m-1}(\cdot, \tau)-\tilde u^{k-1}(\cdot, \tau )
\|^{1-1/s}_{s }\|w_m( \cdot, \tau)\|_{s}\le a_{m,k}\label{E:3.40}
\end{align}
for any $t\in J_{t_{1/2}}$,  where $\sup_{k\ge m}a_{m,k}\ri 0$ as
$m\ri \infty$ and $a_{m,k}$ is decreasing sequence in $m$ for any
fixed $k>m$ (this properties we can obtain arguing in the same way
as in Theorem \ref{T:wp}).
  Using  this with
(\ref{E:Kigna}) and (\ref{E: lucmansahman}), for any $t\in
J_{t_{1/2}}$ we get
\begin{align*}
 K ^{m,k}(t) \le C   \int_0^t ( K ^{m-1,k-1 }(t_1) +K ^{m,k}(t_1))\dd t_1 +a_{m,k}.
\end{align*}
 By the Gronwall  inequality, for any $t\in [0,t_{1/2}]$
we have
\begin{align*}
 & K ^{m+p,k+p}(t) \le C   \int_0^t K^{m+p-1,k+p-1}(\sigma_1)e^{Ct_{1}} \dd \sigma_1 +Ca_{m+p,k+p}  \\ &\le C^2
  \int_0^t\int_0^{\sigma_1} K^{m+p-2,k+p-2}(\sigma_2)e^{C\sigma_{1}}e^{C\sigma_{2}}
\dd \sigma_2\dd \sigma_1
 \\ & \quad
+Ce^{ C t_{1/2}} a_{m+p-1,k+p-1} +Ca_{m+p,k+p}\\
& \le C^3
  \int_0^t\int_0^{\sigma_1}\int_0^{\sigma_2} K^{m+p-3,k+p-3}(\sigma_2)e^{C\sigma_{1}}e^{C\sigma_{2}}e^{C\sigma_{3}}
\dd \sigma_3\dd \sigma_2\dd \sigma_1   \\ & \quad +C\frac{e^{ 2C
t_{1/2}}}{2} a_{m+p-2,k+p-2}+Ce^{ C t_{1/2}} a_{m+p-1,k+p-1}
+Ca_{m+p,k+p}
\\ &  \le C^p \int_0^t\int_0^{\sigma_1} \cdots \int_0^{\sigma_{p-1}}
K^{m,k}(\sigma_p)e^{C\sigma_{1}+C\sigma_{2}+\cdots+C\sigma_p} \dd
\sigma_p\cdots \dd \sigma_2\dd \sigma_1 \\ & \qquad\qquad  +    \sum
_{j=0}^{p-1}C\frac{(e^{ C t_{1/2}})^j }{j!} a_{m+p-j,k+p-j}   .
\end{align*}
 Thus,  we derive
\begin{align*}
  K ^{m+p,k+p}& \le \frac{ C e^{ p C}}{p !} \max_{t\in
  [0,T]}K^{m,k}+ Ca_{m,k} .
\end{align*}
Hence, $w_m$ is a convergent sequence in $C ([0,t_{1/2}],H^{s-1}
(\hat D))$. In the same way we can get the convergence  of $w_m$
in $C ([t_{i-1/2},t_{i+1/2}],H^{s-1} (\hat D))$.   Finally,
the fact   $w_m\in C(\R_+,H^{s-1}  (D) )$ implies that $w_m$
converges to some $w^*$  in $C (\R_+,H^{s-1} ( D) )$.  The
convergence of $w_m$ implies the convergence of $u^m $ to some
$u^*$ in $\XX^ {s}(D)$. We have
\begin{align}
\rot u^*&=  w^*,\label{st3:verj1}\\
 \diver u^* &=0 ,\label{st3:verj2}\\
 u^*(x,t)
\cdot n(x)&=(u_0(x)\mu(t)+ \overline{u}  (x,t  )  n (x ) \text{ on
} \Gamma\times \R_+.\label{st3:verj3}
\end{align}
Let us show that
\begin{align}
 w^*(\cdot,t)&=\sum_{l=i+1}^\infty {w{^*}}^l(\cdot,t)\quad\text{ for }
 t\in[t_{i-1/2},t_{i+1/2}], \label{st3:verj4}
\end{align}
 where ${w{^*}}^l$ is the solution of
\begin{align}
&\p_t{w{^*}}^l+\langle \tilde {u}^*,\nabla\rangle {w{^*}}^l = - (\diver \tilde u^* ){w{^*}}^l\text{ in }\hat {D}\times
\R_+,\\
&{w{^*}}^l(0)=\kappa^l\rot( \pi u_0). \label{st3:verj5}
\end{align}
To this end, recall that
 \begin{align*}
  w_m(\cdot,t)&=\sum_{l=i+1}^\infty w^l_m(\cdot,t),\quad\text{ when }
  t\in[t_{i-1/2},t_{i+1/2}],
  \end{align*}
where $w^l_m$ is the solution of
\begin{align*}
&\dot{w}_m^l+\langle \tilde {u}_{m-1},\nabla\rangle w^l_m =-(\diver \tilde {u}_{m-1}) w^l_m\text{ in }\hat {D}\times \R_+,\\
&w^l_m(0)=\kappa^l\rot( \pi u^{m+1}_0).
\end{align*}
We have that $w^l_m \ri {w{^*}}^l $ in $ C(\R_+, H^{s-1} (\hat
D))$ uniformly with respect to $l$ as $m\ri \infty$  (this can be
proved in the same way as in the proof of the convergence of
$w_m$). Thus we have (\ref{st3:verj4}). Clearly (\ref{st3:verj1})-(\ref{st3:verj5})
imply that $u^*$ is a solution of the Euler system
(\ref{1:1})-(\ref{1:3}).

 As in (\ref{E:wlignahata}), using
(\ref{st3:verj1})-(\ref{st3:verj5}) for any
$t\in[t_{i-1/2},t_{i+1/2}]$ and (\ref{1:expdecay}), we can show
that
\begin{align*}\sum_{l=i}^\infty \| {w{^*}}^l(t)\|_{s-1,\hat D}&\le   C \sum_{l=i}^\infty
\exp(   Ci^2  ) \| \kappa^l\rot( \pi u_0  ) \|_{s-1,\hat D}\\&\le
C\sum_{l=i}^\infty \exp  ( C  i^2)  \exp  (  -Ci^{2+\beta}).
\end{align*}
Thus
\begin{align}\label{E:3.41}
\lim_{t\ri \infty } \|u^*(t)- \overline{u}(t) \|_{s  ,  D}=0.
\end{align}
Combining this with (\ref{E:sbalizard}), we see that the first two terms on the left-hand side of
(\ref{1:karavardy}) go to zero as $t\ri \infty$.
Recall that
\begin{align*}
 \Delta p^*&=-\diver (\langle u^*,\nabla\rangle u^*) \\
 \frac{\p p^*}{\p n}&= -(\langle u^*,\nabla\rangle u^*) \cdot n.
\end{align*}
 Thus,  Proposition \ref{T:poisMaz} implies
$\lim_{t\ri\infty}\|\nabla p^*(t)\|_{s-1}=0$.  This completes the proof of
  Theorem \ref{T:himnth}.

\end{proof}

 \section {Construction of the particular solution}\label{S:3}

\subsection{Proof of Proposition \ref{P:Tetakar} }\label{S:5}
 We have the following simplified version of Proposition~\ref{P:Tetakar}.
\begin{lemma}\label{L:2}For any  $x_0\in   \overline D$ there exist a function
 $ \theta \in C^\infty( [0,1],  \dot H^{s +1}(\hat D))$ and a constant $\lambda>0$ such that
\begin{align}
&\Delta\theta=0 \qquad\text{in}\quad D\times [0,1],\label{2:Thetaharmon}\\
&\frac{\p\theta}{\p n} =0\qquad \text{on}\quad (\Gamma\setminus \Gamma_0)\times [0,1] ,\label{2:Thetaezr}\\
&\supp\theta \subset \hat D\times (0,1),\label{2:Thetakrich}\\
&\phi^{\nabla\theta+(c,0 )}(x_0,1 )\notin\overline D \text{ for any } |c|<\lambda.\label{2:thetaverj1}
\end{align}
\end{lemma}
This lemma is proved at the end of this subsection.
\begin{proof}[Proof of Proposition \ref{P:Tetakar}]
It follows from  Lemma \ref{L:2}   that there  are functions $
\tilde \theta^i \in C^\infty( [0,1],\dot H^{s +1}(\hat D))$ and
open balls $B^i=B(x_i,r_i)\subset \R^2$, $i=1,\ldots,N$  covering
the rectangle $[0,1]\times[-1,1]$ such that properties
 (\ref{1:Thetakrich})-(\ref{1:12}) and  (\ref{1:thetaverj1})    are verified for $\tau_i=1$. For $i=1,\ldots,N$ let us
take
\begin{align}
 \tau_i:&= {i } \sup_{t\in [0,1]}\|{\nabla\tilde
\theta^i(\cdot,t)} \|_{s, \hat D}  ,\label{2:tauinshan}\\
\theta^i  (x,t ):&=\frac{\tilde
\theta^i(x,\frac{t}{\tau_i})}{\tau_i}\label{2:thetainshanak1}.
\end{align}
  Then $B^i, \tau_i$ and $\theta^i$
 verify (\ref{E:3
gnd1})-(\ref{1:thetaverj1}) for $i=1,\ldots,N$. Moreover, there
are closed balls $\tilde B_1, \tilde B_2 \subset \hat D \setminus
\overline D$ such that
\begin{align*}
&{\cup_{i=1}^N \phi^{\nabla\theta^i+(c,0)}(B^i,\tau_i)}\subset
\tilde B_1 \cup \tilde B_2.
\end{align*}
We denote $ B^{2kN+j}:=B   (x_j,r_j)+(k,0)$ and $ B^{
(2k+1)N+j}:=B (x_j ,r_j)-(k+1,0)$, $j=1,\ldots,N$. Then properties
(\ref{E:3 gnd1}) and (\ref{E:3 gnd2}) are satisfied.
 Let $h \in C^\infty([0,1])$ be such that
\begin{align*}
h(t)&= 0\quad\,\,\,\text{for any } t\in[0,1/4], \\
h(t)&=1\quad\,\,\,\text{for any } t\in[3/4,1],\\
 |h(t)| &\le 1\quad\,\,\,\text{for any } t\in[0,1].
\end{align*}
 For
any $x=(x_1,x_2)\in \hat D$ and $c\in \R$ define
\begin{align}
\tilde \theta^{2kN+j}(x,t)&=
\begin{cases}
 (-k-c)x_1 h'(t) & \text{for } t\in[0,1], \\
\tilde \theta^j(x,t-1) & \text{for }  t\in[1,2].
\end{cases}\label{E:taliqverjsd}
\end{align}
 It follows from the constructions of $\tilde \theta^j$,
$j=1,\ldots,N$  that (\ref{1:Thetakrich})-(\ref{1:12}) are
verified for $\tau_i=2$.  It is easy to see  that for any
$t\in[0,1]$ we have \begin{align} \phi^{\nabla\tilde
\theta^{2kN+j}+(c,0)}(x,t)= (-k-c,0)h(t)+ (c,0)t+x.
\label{E:taliqverj}
\end{align} Thus $\phi^{\nabla\tilde
\theta^{2kN+j}+(c,0)}(B^{2kN+j},2)= \phi^{\nabla\tilde \theta^{
j}+(c,0)}(B^{ j},1)\not\subset \overline D$, which implies
(\ref{1:thetaverj1}) and (\ref{1:thetaverj1ar22}).     Notice that
$\nabla\tilde \theta^{i} \in \XX^s(\hat D) $. In order  to have
also (\ref{1:thetaverj2}) and (\ref{1:tauignah}), we   define
$\tau_i  $   by (\ref{2:tauinshan}) and
\begin{align}
\theta^i  (x,t ):&=\frac{2\tilde
\theta^i(x,\frac{2t}{\tau_i})}{\tau_i}\label{2:thetainshanak}.
\end{align}
  This completes the proof.
\end{proof}

\begin{proof}[Proof of Lemma \ref{L:2}] The proof   is based on the ideas  of \cite[Lemma~A.1]{CorNS}.

\vspace{6pt}  \textbf{Step 1.} We denote by $\aA$ the vector space
of functions $ \xi \in \dot H^{s+1}( \hat D )$ with the following
properties
\begin{align}
 & \Delta\xi=0  \text{ in }  D,\nonumber\\
   &\frac{\p\xi}{\p n} =0  \text{ on }   \Gamma\setminus \Gamma_0,
   \nonumber\\
   &\supp \xi \subset \hat D.
\end{align}
First, let us show that for any $x_0\in \overline D$ we have
\begin{align}
\R^2=   \{&\nabla\xi(x_0):\,\xi\in \aA\}.
\label{2:R2inerkay}
\end{align}
 Suppose that (\ref{2:R2inerkay}) does not hold. Then, there is a  vector $V\in\R^2$, $ V\neq 0$ such that
$$V\cdot\nabla\xi(x_0)=0 $$
for all $\xi \in \aA$. Let $ \tilde D$ be the domain defined in
(\ref{0:0:dtildisahm}) and let $ \tilde D\subset D_1$. Take any
$a\in \tilde D\setminus \overline D$, and let $G_a$ be the
solution of (\ref{0:G1}), (\ref{0:G2}). Let $ B_1,B_2 \subset
\tilde D\setminus  \overline D  $ be two open neighborhoods of $a$
such that $ \overline B_1 \subset   B_2$ and let $ \rho\in
C^\infty(\tilde D) $ be such that
\begin{align*}\rho(x)=
\left\{
  \begin{array}{ll}
    1, & \hbox{\text{if} $x\notin \overline {B_2}$,} \\
    0, & \hbox{$x\in B_1$.}
  \end{array}
\right.
\end{align*}
Clearly $ \pi (\rho G_a )\in \aA$, thus $ V\cdot \nabla\pi( \rho
G_a) (x_0) =0$. Since $ x_0 \notin \overline B_2$, we have
    \begin{align}\label{2:Vangama}
V\cdot \nabla G_a(x_0) =0
\end{align}
for all $ a \in \tilde D\setminus \overline D$. On the other hand  $ G_a$ is analytic in
 $a\in \tilde D \setminus \{x_0\}$ (see Proposition \ref{T:Green}, (iii)). Thus,
 we have  (\ref{2:Vangama})   for all $a \in \tilde D\setminus \{x_0\}$.  Using
(\ref{0:grgnahat2}), one can find a sequence $a_n\ri x_0$ such that
$V\cdot\nabla G_{a_{n}} (x_0 )\ri\infty $ as $n\ri \infty$, which
is a contradiction to $V\neq0$.

\vspace{6pt}  \textbf{Step 2.}  Take any $x_0\in D\cup \Gamma_0$,
$x^1\in\hat D\setminus \overline D$ and let   $F:   [0,1] \to
 \hat D $ be a continuous function such that
\begin{align}
F(t)&=x_0\quad\,\,\,\,\,\,\text{for any } t\in[0,1/4],\nonumber\\
F(t)&=x^1\quad\,\,\,\,\,\,\text{for any }
t\in[3/4,1],\nonumber\\
F(t)&\notin \Gamma\setminus\Gamma_0\,\,\,\text{for any } t\in
[0,1]\nonumber.
\end{align}
Then for any $\e>0$ we can find $\xi_i\in \aA$, $h_i\in
C^\infty ([0,1])$, $i=1,\ldots,k$   with $\supp h_i\subset
[1/4,3/4]$ such that for
$\theta(x,t):=\sum_{i=1}^k\xi_i(x)h_i(t)$ we have
\begin{align}\label{2:F(t)}
|F(t)-\phi^{\nabla\theta}(x_0,t )|<\e
\end{align}
for any $t\in [0,1]$. It is easy to see that there is a constant
$\lambda>0$ such that for any $|c|<\lambda$
\begin{align}\label{2:phi}
|\phi^{\nabla\theta}(x_0,t )-\phi^{\nabla\theta+(c,0)}(x_0,t
)|<\e.
\end{align}
  Since   $\xi_i\in \aA$  and $\supp h_i\subset
[1/4,3/4]$, we have (\ref{2:Thetaharmon})-(\ref{2:Thetakrich}).
 The construction of $F$, inequalities
(\ref{2:F(t)}) and (\ref{2:phi}) imply $\phi^{\nabla\theta+ (c,0 )} (x_0,1
)\notin~\overline D$ for sufficiently small $\e>0$.

\vspace{6pt}  \textbf{Step 3.} It remains to study the case
$x_0\in \Gamma\setminus\Gamma_0$. Let $y_0\in \Gamma_0$ and $k\in
\R$ be such that $x_0=y_0+(k,0)$. Then, the function
\begin{align*}
  \theta (x,t)&=
\begin{cases}
 (-c-k) x_1 h'(t) & \text{for } t\in[0,1/2], \\
  2\theta_{y_0}(x,2(t-1/2)) & \text{for }  t\in[1/2,1]
\end{cases}
\end{align*}
satisfies (\ref{2:Thetaharmon})-(\ref{2:thetaverj1}), where $h\in C^\infty([0,1/2])$ is any function with $ h(0)=0,$ $h(1/2)=1$ and $\theta_{y_0}$ is the function constructed in Step 2 for $y_0\in\Gamma_0$.
\end{proof}

\subsection{Proof of Proposition \ref{L:flow}}
For any $m\in \R_+$, let us denote
\begin{align}
D^m_-:=(-\infty,-m]\times[-2,2] \text{ and } D^m_+:=[m,+\infty)\times [-2,2].
\end{align}
We shall need the following lemma.
\begin{lemma}\label{L:flow11}  The functions $ \theta^i  $ constructed in the proof of  Proposition \ref{P:Tetakar} are such that there exist  $\ph^i\in C(\R_+)$ with
\begin{align}
&\sup_{x\in\overline D }|\phi^{\nabla\theta^i + (c,0 ) }
(x,t)-x|\le \bigg[ \frac{i}{2N}\bigg]+M \text { for any }
t\in[0,\tau_i]\label{1:thetaverj22},\\
 & |\nabla\theta^i(x,t)-\nabla\theta^i(y,t)| \le\frac{\ph^i(t)}{(m+1)^2} |x-y| \text{ for any }x ,y\in D^m_+  \text { or } x,y\in D^m_- ,
  \label{1:thetaverj}\end{align} where  $ \int _0^{\tau_i}\ph^i(t)\dd
t \le M $, $ [ \frac{i}{2N} ]$ is the integer part of $
\frac{i}{2N}$ and $M\in \R$ does not depend on $i$.
\end{lemma}
\begin{proof}It is easy to see that  (\ref{E:taliqverjsd}) and (\ref{2:thetainshanak}) imply
\begin{align*}
\phi^{\nabla\theta^i + (c,0 ) } (x,t)&=
\begin{cases}
 (-k-c,0)h(\frac{2t}{\tau_i})+ (c,0)\frac{2t}{\tau_i} +x& \text{for } t\in[0,\tau_i/2], \\
\phi^{\nabla\tilde \theta^j+(c,0)}(x,\frac{2t}{\tau_i}-1) & \text{for }
t\in[\tau_i/2,\tau_i].
\end{cases}
\end{align*}
where $k=\big[ \frac{i}{2N}\big]$ and $j= i-2Nk$. This yields
(\ref{1:thetaverj22}) for a sufficiently large $M$. To prove
(\ref{1:thetaverj}), notice that in the proof of Lemma \ref{L:2},
the functions $\theta$ can be chosen  such that
$$\| x_1^2\p^\beta\theta \|_{s,\hat D}<C(x_0), $$ where
$|\beta|=2$. Indeed, since Proposition \ref{T:Green} implies that
the second order derivatives of $G_a$ belong  to $\sS(D)$, one can
replace (\ref{2:R2inerkay}) by
\begin{align*}
\R^2=   \{&\nabla\xi(x_0):\,\xi\in \aA\text{ and } \|
x_1^2\p^\beta\xi \|_{s,\hat D}<C(x_0), \quad |\beta|=2 \}.
\end{align*}
Hence, we can find a constant $M_1$ such that
\begin{align*} \sup_{i=1,\ldots,N, |\beta|=2 }\int_ 0^ {1}\|x_1^2\p^\beta \tilde \theta^i(t,\cdot) \|_{L^\infty
(\hat D)}\dd t <M  _1  .
\end{align*}
 Combining this with (\ref{E:taliqverjsd})  and (\ref{2:thetainshanak}), we get (\ref{1:thetaverj}).
\end{proof}

Now we  return to the proof of Proposition \ref{L:flow}. It
suffices to show that for any $\e>0$ there is   $\nu>0$ such that
the inequality
 \begin{align} \label{E:Lemiaraj}
 \sup _{x\in B_i}|\phi^{ {u}}(x,t )-\phi^{\overline{u}}(x,t) |\le \e
\end{align}
holds for any    $i\ge 1$ and $t\in[0,t_{i-1/2}]$.
Let us denote
\begin{align*}
X(t)=\phi^{ {u}}(x,t) , \\
Y(t)=\phi^{\overline{u}}(x,t) ,
\end{align*}
where $x\in B^i  $. We shall prove (\ref{E:Lemiaraj}) in the case
when $i$ is even. The proof when $i$ is odd is similar. Let
$k:=\left[\frac{i}{2N}\right]$, then
 \begin{align}\label{1:lem2.312}
   B^i \subset [k-2,k +3]\times [-2,2].
\end{align}
\vspace{6pt}  \textbf{Step 1.}
 First let us show that to establish  (\ref{E:Lemiaraj}) it suffices  to prove that
\begin{align}\label{1:anhra11}
|X(t )-Y(t )|<1 \text { for all } t\in \R_+.
\end{align}
 It is easy to see that
 \begin{align}\label{1:lem2.2}
 &\p_t\left( X(t )-Y(t ) \right)=  {u}(X(t ),t)-\overline{u}(Y(t ),t)\nonumber\\&
 = ({u}(X(t ),t)-\overline{u}(X(t ),t))+(\overline{u}(X(t ),t)-\overline{u}(Y(t ),t))=:I_1(t)+I_2(t).
\end{align}
We have that
\begin{align}\label{1:lem2.44}
\int_0^\infty |I_1(t)|\dd t\le \nu.
\end{align}
 From (\ref{E:Tgcikar1}), (\ref{E:Tgcikar2}),
(\ref{1:thetaverj22}) and (\ref{1:lem2.312}) it follows that
$$ Y(t )\in [k -2-\left[\frac{j} {2N}\right]- M,k +3+\left[\frac{j} {2N}\right]+M]\times[-2,2]$$
for any $t\in[0,t_{j-1/2}]$. Hence, (\ref{1:anhra11}) implies
$$ X(t )\in [k -3-\left[\frac{j} {2N}\right]- M,k +4+\left[\frac{j} {2N}\right]+M]\times[-2,2].
$$
We derive from (\ref{E:Tgcikar1}), (\ref{E:Tgcikar2})
 and (\ref{1:thetaverj})
that
 \begin{align}
 \int_0^{t_{i-1/2}}| I_2(t)| \dd t
\le \int^{t_{i-1/2}}_{0}  \Psi(t)| X(t )- Y(t)| \dd
t,\label{E:4.17}
\end{align}
where
\begin{align*}\Psi(t)=
\left\{
  \begin{array}{ll}
   \frac{ \ph^j(t-t_{j-1})}{(k -2-\left[\frac{j} {2N}\right]- M)^2}, & \hbox{\text{} $t\in[t_{j-1},t_{j-1/2}]$,} \\
     \frac{ \ph^j(t_j-t)}{(k -2-\left[\frac{j} {2N}\right]- M)^2}, &\hbox{\text{} $t\in[t_{j-1/2},t_{j}]$ }
  \end{array}
\right.
\end{align*}
for $j<2N(k-3-M)$ (here we use (\ref{1:thetaverj}) for $m=k
-3-\left[\frac{j} {2N}\right]- M$) and
\begin{align*}\Psi(t)=
\left\{
  \begin{array}{ll}
  \ph^j(t)  , & \hbox{\text{} $t\in[t_{j-1},t_{j-1/2}]$,} \\
  \ph^j(t_j-t)  , &\hbox{\text{} $t\in[t_{j-1/2},t_{j}]$ }
  \end{array}
\right.
\end{align*}
for $ j\ge 2N(k-3-M)$ (in this case we use (\ref{1:thetaverj}) for
$m=0$). Thus we have
 \begin{align}
  \int^{t_{i-1/2}}_{0}  \Psi(t) \dd
t = &\int^{t_{2N(k-3-M)-1}}_{0}  \Psi(t) \dd t
+\int^{t_{i-1/2}}_{t_{2N(k-3-M)-1}} \Psi(t) \dd t \nonumber\\ \le&
\sum_{j=1}^{2N(k-3-M)-1}\frac{2M}{(k -2-\left[\frac{j}
{2N}\right]- M)^2}+(2N(M+4)+1)2M.\label{E:4.127}
\end{align}
 Integrating (\ref{1:lem2.2}), using
(\ref{1:lem2.44})-(\ref{E:4.127}) and the Gronwall inequality, we
obtain
 \begin{align}
| X(t_{i-1/2} )&-Y(t _{i-1/2}) |\nonumber\\& \le \nu
\exp\left(\sum_{j=1}^{2N(k-3-M)-1}\frac{2M}{(k -2-\left[\frac{j}
{2N}\right]- M)^2}+(2N(M+4)+1)2M\right)\nonumber\\& \le \nu
\exp\left(\sum_{j=1}^{\infty}\frac{8MN^2}{j^2}+(2N(M+4)+1)2M\right).\label{E:4.128}
\end{align}
Choosing $\nu$ such     that the right-hand side of
(\ref{E:4.128}) is smaller  than $\e$, we prove (\ref{E:Lemiaraj})
for all $i$.

 \vspace{6pt}  \textbf{Step 2.} To complete the proof, it remains
 to show (\ref{1:anhra11}). To this end, let us  assume that (\ref{1:anhra11}) does not hold for some $t>0$. Denote by $\tilde t_0$ the first
 time
 such that $|X(\tilde t_0)-Y(\tilde t_0)|=1$. Hence, we have (\ref{1:anhra11}) for all $t<\tilde t_0$. Step
 1 implies
 \begin{align}\label{E:4.129}
| X(\tilde t_0)&-Y(\tilde t_0) | \le \nu
\exp\left(\sum_{j=1}^{\infty}\frac{8MN^2}{j^2}+(2N(M+4)+1)2M\right).
\end{align}
Since the right-hand side of (\ref {E:4.129}) does not depend on
$\tilde t_0$, choosing $\nu$ sufficiently small, we get
(\ref{1:anhra11}).

\section{Appendix: proof of Lemma \ref{L:rotdivihamar}} \label{S:0:3}
 Let us consider the  space
$$ H_0(D')=\{z\in L^2(D') :\,\rot z \in L^2(D'), \,\diver z\in L^2(D'),\, z\cdot n|_{\Gamma'}=0
\}
$$  endowed  with the norm $$ \|z\|_{H_0}=\|z\|+\|\rot z\|+\|\diver z\|.
$$
Here $D'$ is a strip   or is the domain $\tilde D$ defined in
(\ref{0:0:dtildisahm}).
 Recall the following result (see \cite[Chapter 7, Theorem 6.1]{Duvlions}).
\begin{theorem}\label{Duvlions} The following equality holds $$\{ z\in H^1(D'):\,  z\cdot n|_{\Gamma'}=0 \} =H_0.$$
\end{theorem}
 In
the case of bounded domains it is shown in \cite[Appendix 1,
Proposition 1.4]{Tembook} that
\begin{align}\label{0:temamcyl}
H^s(\Omega)=\{z\in L^2(\Omega):\, \rot z\! \in\!
H^{s-1}(\Omega),\,\diver z\!\in\! H^{s-1}(\Omega), \, z\cdot n \in
H^{s-1/2}(\p \Omega) \}.
\end{align}
Let us  generalize this result to the case of domain $D'$. We
shall need the following lemma.
\begin{lemma}\label{App:Lemma}
Let $g\in H^{1/2}(\Gamma')$. Then the  problem
\begin{align}
\Delta u-u=0\,\,\, &\text{ in}\quad D',\label{Ap:111}\\
\frac{\p u}{\p n} =g  \,\,\,&\text{ on} \quad
\Gamma' \label{Ap:211}
\end{align}
has a unique solution $u\in H^2(D')$, which satisfies
\begin{align}\label{Ap:anhav} \|u\|_{2}\le C\|g\|_{1/2}.
\end{align}
\end{lemma}
\begin{proof} Problem  (\ref{Ap:111}), (\ref{Ap:211}) is equivalent to
\begin{align*}
\int_{D'} \nabla u \nabla \theta \dd x+\int_{D'}  u  \theta \dd x =\int_{\Gamma} g  \theta \dd \sigma\,
\text { for any }\theta \in H^1(D'). \end{align*}
Since $g\in H^{-1/2}(\Gamma')$, the Riesz representation
theorem implies the existence of a unique solution $u \in   H^1(D') $.

\vspace{6pt}  \textbf{Case 1.} Assume   $D'=D$, and let us prove  that $ u\in H^2(D)$. It is easy to see that $ v:= \p_1 u$ is the
solution of the problem
\begin{align*}
\Delta v-v=0\,\,\, &\text{ in}\quad D, \\
\frac{\p v}{\p n} = \p_1g  \,\,\,&\text{ on} \quad
\Gamma .
\end{align*}
Thus $\p_1 u \in H^1(D)$ and
$$ \|\p_1u\|_{1}\le C\|g\|_{1/2}.$$
Combining this with the fact that $ \Delta u \in H^1(D)$, we obtain $u\in H^2(D)$ and~(\ref{Ap:anhav}).

\vspace{6pt}  \textbf{Case 2.} Now consider the case $D'=\tilde D$. Let
\begin{align*}
\Omega_1:=\{ x\in \tilde D: \, |x_1|< N\} \text{ and } \Omega_2:=\{ x\in \tilde D: \, |x_1|< N+1\},
\end{align*}
where $N$ is so large that
 $\tilde D\setminus D \subset \Omega_1 $. Let us take some function $\chi\in C ^ \infty(\overline {\tilde {D} })  $ such that
 such that
\begin{align*}\chi(x)=
\left\{
  \begin{array}{ll}
    0, & \hbox{\text{if} $x\notin \overline\Omega_{2}$,} \\
    1, & \hbox{\text{if} $x\in\Omega_{
    1}$.}
  \end{array}
\right.
\end{align*}
 Then $ w:=\chi u $ is the solution of
 \begin{align}
\Delta w-w= 2 \nabla \chi\nabla u+\Delta \chi u=:f\,\,\, &\text{ in}\quad \Omega_2,\label{Ap:11}\\
\frac{\p w}{\p n} =:\tilde g  \,\,\,&\text{ on} \quad
\p \Omega_2 .\label{Ap:21}
\end{align}
It is easy to see that $f\in L^2(\Omega_2)$  and $\tilde g \in H^{1/2}(\p \Omega_2)$. This implies that $ w ~\!\in~\! H^2(\Omega_2)$
(e.g., see \cite{ADN}). Thus $u\in H^2(\Omega_1)$. On the other hand,
 from the fact  $\Gamma_0\subset \Omega_1 $ we derive $ \frac{\p u}{\p n}|_\Gamma\in H^{1/2}(\Gamma)$. Hence, using the result for $D'=D$, we see that
 $u\in  H^2(D)$.
 This completes the proof of Lemma \ref {App:Lemma}.
\end{proof}

Now let us prove (\ref{0:temamcyl}) for $\Omega=D'$. Clearly the
space in the left-hand side  is contained in the right-hand side
of (\ref{0:temamcyl}). By induction, let us show the other
inclusion. Assume $s=1$. Let us take some function $z$ from the
right-hand side of (\ref{0:temamcyl}) and   consider the problem:
\begin{align*}
\Delta p-p=0\,\,\, &\text{ in}\quad D, \\
\frac{\p p}{\p n}   =z\cdot n \,\,\,&\text{ on} \quad \Gamma.
\end{align*}
By Lemma \ref {App:Lemma}, we have $p\in H^2(D')$ and  $\|p\|_{2}\le
C  \|z\cdot n\|_{1/2}$ . Let us take  $w=z-\nabla p$.
Clearly $w\in H_0$, thus Theorem \ref{Duvlions} implies $w\in
H^1(D')$. Hence, $ z\in H^1(D')$ and
\begin{align}\label{0:1:3}
\|z\|_1\le  \|w\|_1+ \|p\|_2\le C(\|z\|+\|\rot z\|+\|\diver
z\|+\|z\cdot n\|_{1/2}).
\end{align}
Now assume that (\ref{0:temamcyl}) holds for $s-1$ and let us
prove it
  for $s$. Let $\tilde n$ be a  regular extension of $n$ in $D'$  such that $|\tilde n(x)|~\!=~\!1$.
Let us show that such an extension exists. To simplify the proof,
let us assume that  $d=0$ in the definition of $\tilde D$ (see
(\ref{0:0:dtildisahm})).  We define
\begin{align*}
\tilde n_1(x_1,x_2)&= -\frac{\gamma'(x_1)}{\sqrt{1+\gamma'(x_1)^2}}+h(x_1,x_2),\\
 \tilde
 n_2(x_1,x_2)&=\frac{x_2}{(1+\gamma(x_1))\sqrt{1+\gamma'(x_1)^2}},
\end{align*}
where $ h\in C_b^\infty(\overline{\tilde {D}})$,  $h|_{\p \tilde
D}=0$ and $ h(x_1,0)=1+
\frac{\gamma'(x_1)}{\sqrt{1+\gamma'(x_1)^2}}$. Then we have
$(\tilde n_1,\tilde n_2)|_{\p \tilde D} =n$ and $ |(\tilde
n_1,\tilde n_2)|>\delta$ for sufficiently small $\delta>0$. Hence,
$\tilde n(x)= \frac{(\tilde n_1,\tilde n_2)}{|(\tilde n_1,\tilde
n_2)|}$ is an extension of $ n$.
   Let us take $v:=\nabla^\bot(z\cdot\tilde n) $.
   Then $v\in L^2,$  $\diver v=0 $. Since  $v\cdot \tilde n $ is   the tangential derivative  of $z\cdot\tilde n $ along $\Gamma'$,
    we have $v\cdot \tilde n\in H^{s-3/2}(\Gamma') $.
   On the other hand
\begin{align*}
-\curl v=\Delta (z\cdot \tilde n) = (\Delta z_1 ) \tilde
n_1+(\Delta z_2) \tilde n_2+\tilde v,
\end{align*}
where $ \tilde v \in H^{s-2}$. It follows from the facts $\Delta z_1= \p_1\diver z+\p_2\curl z $
and $\Delta z_2= \p_2\diver z-\p_1\curl z $ that $\curl v\in H^{s-2}$.
   Thus the induction hypothesis yields $\nabla^\bot(z\cdot~\!\tilde n) ~\!\in ~\!H^{s-1}$. Hence,
  \begin{align*}
(\p_2 z_1)\tilde n_1+(\p_2 z_2)\tilde n_2\in H^{s-1},\\
(\p_1 z_1)\tilde n_1+(\p_1 z_2)\tilde n_2\in H^{s-1}.
\end{align*}
  Combining this with  $ \diver z\in H^{s-1}$ and $ \rot z\in H^{s-1}$,  we obtain  $\tilde n \cdot \nabla^\bot z_i \in H^{s-1} $ and
  $\tilde n \cdot \nabla z_i \in H^{s-1} $ for $i=1,2$. Thus $\nabla z_i \in H^{s-1}$, which completes the proof.

\end{document}